\documentclass[final,1p,times]{elsarticle}

\usepackage{amssymb}
\usepackage{amsmath}
\usepackage{graphicx}
\usepackage{xspace}
\usepackage[ruled]{algorithm2e}

\journal{Communications in Nonlinear Science and Numerical Simulation}

\newcommand{\Fix}{\mathrm{Fix}}
\newcommand{\tagname}{touch-and-go\xspace}
\newcommand{\isoname}{symmetry breaking\xspace}
\newcommand{\Isoname}{Symmetry breaking\xspace}
\newcommand{\PM}{\mathcal{P}}
\newcommand{\xfp}{x_{fp}}
\newcommand{\muSJ}{\mu_{SJ}}
\newcommand{\muEM}{\mu_{EM}}

\newtheorem{theorem}{Theorem}
\newtheorem{lemma}[theorem]{Lemma}
\newtheorem{remark}[theorem]{Remark}
\newtheorem{definition}[theorem]{Definition}

\newenvironment{proof}{\textbf{Proof:}}{\qed}

\begin{document}

\begin{frontmatter}
\title{Validated numerics for period-tupling and \tagname bifurcations of symmetric periodic orbits in reversible systems}

\author{Irmina Walawska\corref{author1}}
\ead{Irmina.Walawska@ii.uj.edu.pl}

\author{Daniel Wilczak\corref{author2}}
\ead{Daniel.Wilczak@ii.uj.edu.pl}

\address{Faculty of Mathematics and Computer Science,
Jagiellonian University, \L ojasiewicza 6, 30-348 Krak\'ow, Poland.}
\cortext[author1] {The work of the first author was supported by the Polish National Science Center under grant UMO-2016/23/N/ST6/02006.}
\cortext[author2] {Corresponding author. This research is partially supported by 
the Polish National Science Center under Maestro Grant No. 2014/14/A/ST1/00453 and under Grant No. 2015/19/B/ST1/01454.}
\date{\today}
\begin{abstract}
We propose a general framework for computer-assisted verification of the presence of \isoname, period-tupling and \tagname bifurcations of symmetric periodic orbits for reversible maps. The framework is then adopted to Poincar\'e maps in reversible autonomous Hamiltonian systems.

In order to justify the applicability of the method, we study bifurcations of halo orbits in the Circular Restricted Three Body Problem. We give a~computer-assisted proof of the existence of wide branches of halo orbits bifurcating from $L_{1,2,3}$-Lyapunov families and for wide range of mass parameter. For two physically relevant mass parameters we prove, that halo orbits undergo multiple period doubling, quadrupling and third-order \tagname bifurcations.
\end{abstract}

\end{frontmatter}

\section{Introduction}
In the past 40 years there were proposed very efficient methods for numerical continuation of  periodic orbits for general ODEs, or periodic orbits satisfying certain symmetries with special focus on applications to Hamiltonian systems \cite{DRPKDGV,MFGDV,H}. They are implemented in very efficient packages, such as AUTO \cite{D}, MATCONT \cite{DGK} and CONTENT \cite{KL}.

Most of the methods are similar in the spirit. The family of periodic orbits satisfies certain finite-dimensional implicit equations, which is solved by a Newton-like scheme. The equations usually involve period (return  time) of the orbit, space variables and/or value of the Hamiltonian. The Newton-like iteration requires computation of monodromy matrix, which is not an issue nowadays, where many efficient tools for integration of ODEs along with variational equations are available \cite{ABBR,JZ,CAPD}.

Bifurcations of periodic orbits can be detected by looking at changes of determinant of monodromy matrix \cite{WS}, inspecting so-called stability parameter in low dimensional systems \cite{GM,H}, analysis of normal form after Lyapunov--Schmidt reduction \cite{CV} or solving an equation specific for the type of bifurcation \cite{NS}. Literature on the topic is really wide, and the list of methods and references mentioned above is clearly incomplete. 

The present paper is complementary to the above results and methods. We focus on reversible maps and reversible Hamiltonian systems. The primary result of the paper is a general framework for rigorous computer-assisted verification, that a branch of symmetric periodic orbits undergoes period-tupling and/or \tagname bifurcation. Finding an approximate numerical candidate for bifurcation point is just a preliminary step of the validation algorithm and for this purpose we can use any of the methods mentioned above \cite{CV,H,GM,NS,WS}. Then, checking several inequalities on the (Poincar\'e) map under consideration and its derivatives on an explicit neighbourhood (input to the algorithm) of the approximate bifurcation point we can prove, that it contains a~bifurcation point and we can conclude about type of bifurcation. As an output of the algorithm we obtain guaranteed bounds on both the bifurcation point and the parameter of the system, at which bifurcation occurs.

In order to justify applicability of the framework we apply it to the Circular Restricted Three Body Problem (CR3BP) \cite{S}. We give a computer-assisted proof of the existence of wide branches of so-called halo orbits bifurcating from the families of planar Lyapunov orbits around $L_{1,2,3}$ libration points. We also prove, that for some physically relevant mass parameters of the system, these branches undergo period-tupling and \tagname bifurcations. These rigorous results are justification of some numerical observations from previous articles, in particular \cite{GM,DRPKDGV}. We also would like to emphasize, that we have found a new phenomenon regarding bifurcations of halo orbits near $L_3$ libration point, when the mass parameter of the system tends to zero. We have observed, that the energy at the \isoname bifurcation, which creates halo orbit, as a function of the mass parameter is not monotone, when the relative mass tends to zero --- see Remark~\ref{rem:L3monotonocity}. Although not validated, this observation has been made using nonrigorous numerics with very high accuracy. 

Rigorous numerical investigation of periodic orbits to ODEs became quite standard  \cite{Ar,BW,C,K,KZ,Pi,SK,W, WZ1,WZ2}. To the best of our knowledge, there are very few results regarding computer-assisted verification of bifurcations of periodic orbits of ODEs, and the field remains widely open. Validation, that a family of  periodic orbits undergoes a bifurcation usually involves computation of rigorous bounds on higher order derivatives of the trajectories with respect to initial condition (except some special cases when the system admits additional structure). The algorithm capable to do that appeared just in 2011 \cite{WZ4} (implemented as a part of publicly available CAPD library \cite{CAPD}) and to the best of our knowledge there are no other publicly available algorithms for rigorous integration of higher order variational equations. The $\mathcal C^r$-Lohner algorithm from \cite{WZ4} has been already applied to study period-dubling bifurcations of periodic orbits \cite{WZ3} in the R\"ossler system \cite{R}, homoclinic tangencies of periodic orbits in a time-periodic forced-damped pendulum equation \cite{WZ5} and non-local cocoon bifurcations \cite{KWZ} in the Michelson system \cite{Mich}. Very recently \cite{BLJ}, another approach to compute periodic orbits for ODEs without integration of the system has been proposed. Periodic solutions are approximated via piecewise Chebyshev polynomials and then their existence is validated by analysis of certain nonlinear operator on a Banach space of Chebyshev coefficients. The efficiency of the method is illustrated on the example of Equilateral Circular Restricted Four Body Problem. Similar approach has been successfully applied for validated computation of bifurcations of equilibria for ODEs and steady states for PDEs --- see for example \cite{BLV,L16,LSW}. A different, geometric method for validation of bifurcations of steady states in the Kuramoto-Sivashinsky PDE is also proposed in \cite{Z}.

Our algorithm for proving the existence of period-tupling and \tagname bifurcation is similar in the spirit to that proposed in \cite{WZ3,Z}, but exploits the presence of a reversing symmetry of the system. After fixing an appropriate coordinate system in a neighbourhood of an apparent bifurcation points, we perform validated Lyapunov-Schmidt reduction. In this way, the analysis of bifurcation is transformed to zero-finding problem of a bivariate scalar-valued function, called \emph{bifurcation function}. Then, checking some inequalities on derivatives of the bifurcation function we can prove, that its set of zeroes is the union of two smooth curves which intersect at an unique point --- the bifurcation point. Finally, some non-degeneracy conditions (inequalities on higher order derivatives of the bifurcation function) let us to conclude about the type of bifurcation. In this paper we restrict to \isoname, period-tupling and \tagname bifurcations.

The article is organized as follows. In Section~\ref{sec:notation} we introduce notation and main definitions used in the paper. Theoretical results, which constitute a basis of the computational framework for general reversible (Poincar\'e) maps are presented in Section~\ref{sec:theory} and then adopted to autonomous reversible Hamiltonian systems in Section~\ref{sec:hamiltonian}. Finally, in Section~\ref{sec:cr3bp} we give formal statements of theorems regarding continuation and bifurcations of halo orbits in the CR3BP.  

In the Appendix we present two auxiliary yet new results, which are included mainly for self-consistency of the paper. Frequently we have to show, that a solution to an implicit equation is defined over an explicit domain. This step appears for example in continuation of periodic orbits or in validated Lyaunov-Schmidt reduction. Although there are available methods for this purpose (see for example \cite{BLM,GLM,WZ4,BW}), in~\ref{sec:continuation} we provide an improvement, which takes advantage from higher order derivatives and flattening the implicit function by a smooth, well chosen substitution. In~\ref{sec:findingBifurcationPoints} we propose own and short algorithm for finding an approximate bifurcation point. It takes advantage from the Lyapunov--Schmidt reduction and computation of higher order derivatives of the (Poincar\'e) map under consideration. Using it, we could easily localize approximate bifurcation points in the CR3BP with the accuracy $10^{-60}$. 

\section{Preliminaries}\label{sec:notation}
\subsection{Notation and basic definitions}
For a map $f\colon D\subset X\to X$, a predicate $\mathcal C\colon X\to \{\texttt{true},\,\texttt{false}\}$ and a set $U\subset D$ we introduce the following notation
\begin{equation*}
 \Fix(f,U,\mathcal C) =\{ x\in U : \left(f(x)=x\right)\wedge \mathcal C(x)\}.
\end{equation*}
Thus, $\Fix(f,U,\mathcal C)$ is the set of fixed points of $f$ in $U$ satisfying constraint $\mathcal C$. Although there is a natural correspondence between subsets of $X$ and univariate predicates defined on $X$, we will separate $U$ and $\mathcal C$ to emphasize rather rare property $\mathcal C$ in a larger (usually open) set $U$. We will also write $\Fix(f,U)$ if $\mathcal C(x)\equiv \texttt{true}$. 
For a set $M\subset \mathbb R\times X$ and $\nu\in\mathbb R$, we define its slice by $M_\nu=\{x\in X : (\nu,x)\in M\}$. For a map $f\colon M\subset \mathbb R\times X\to X$ and fixed $\nu\in\mathbb R$ we define $f_\nu\colon M_\nu\to X$ by $f_\nu(x)=f(\nu,x)$.

\begin{definition}
 A homeomorphism $R \colon X\to X$ is called a reversing symmetry for $f\colon M\subset X\to X$ if $(R\circ f)(M)\subset M$ and for $x\in M$ there holds 
 \begin{equation*}
  (R^{-1}\circ f\circ R\circ f)(x) = x.  
 \end{equation*}
\end{definition}
It is easy to see that any reversing symmetry $R$ for $f$ is also a reversing symmetry for all iterations $f^n$ defined on their proper domains.  

\begin{definition}
 A homeomorphism $S \colon X\to X$ is called a symmetry for $f\colon M\subset X\to X$ if $S(M)\subset M$ and for $x\in M$ there holds
 \begin{equation*}
  S(f(x)) = f(S(x)).
 \end{equation*} 
\end{definition}

For a map $S\colon X\to X$ we define a predicate $\mathcal C_S\colon X\to\{\texttt{true},\,\texttt{false}\}$ by
\begin{equation}\label{eq:predicateSymmetry}
\mathcal C_S(x) = \left(x\in\Fix(S,X)\right).
\end{equation}
We will use this notation to select points satisfying certain symmetries. 

\subsection{Geometric definitions of two types of bifurcations}
Our primary object of interest is a $\mathcal C^3$-smooth function $f\colon M\to X$ defined on an open set $M\subset \mathbb R\times X$, where $X$ is a smooth manifold. In this article we make the following standing assumption:
\begin{itemize}
 \item[\textbf{C1:}] for $\nu\in\mathbb R$, if $M_\nu\neq \emptyset$ then the function $f_\nu\colon M_\nu\to X$ is a diffeomorphism onto image.
\end{itemize}
The above assumption fits the applications we have in mind, i.e. $f_\nu$ will be a one-parameter family of Poincar\'e maps. Thus, the domain of each map $f_\nu$ may vary with the parameter $\nu$. 

Bifurcations are usually defined by their normal forms \cite{CH}. The following 
\begin{eqnarray}
 f^k_\nu(x) &=& x(\nu-x^2),\label{eq:normalFormPT}\\
 f^k_\nu(x) &=& x(\nu-x)\label{eq:normalFormTAG}
\end{eqnarray}
describe period-tupling and transcritical bifurcations, respectively. When an eigenvalue of the derivative of a reversible planar map at a symmetric fixed point crosses $1:k$ resonance, $k\geq 2$, generically period-tupling or \tagname bifurcation occurs (notation following \cite{BBS}) --- two types of generic bifurcations for strong resonances are illustrated on Figure~\ref{fig:resonances}. The branches of symmetric periodic points near bifurcation point look similarly to those for pitchfork and transcritical bifurcations, respectively. 
\begin{figure}
\centerline{\includegraphics[width=.9\textwidth]{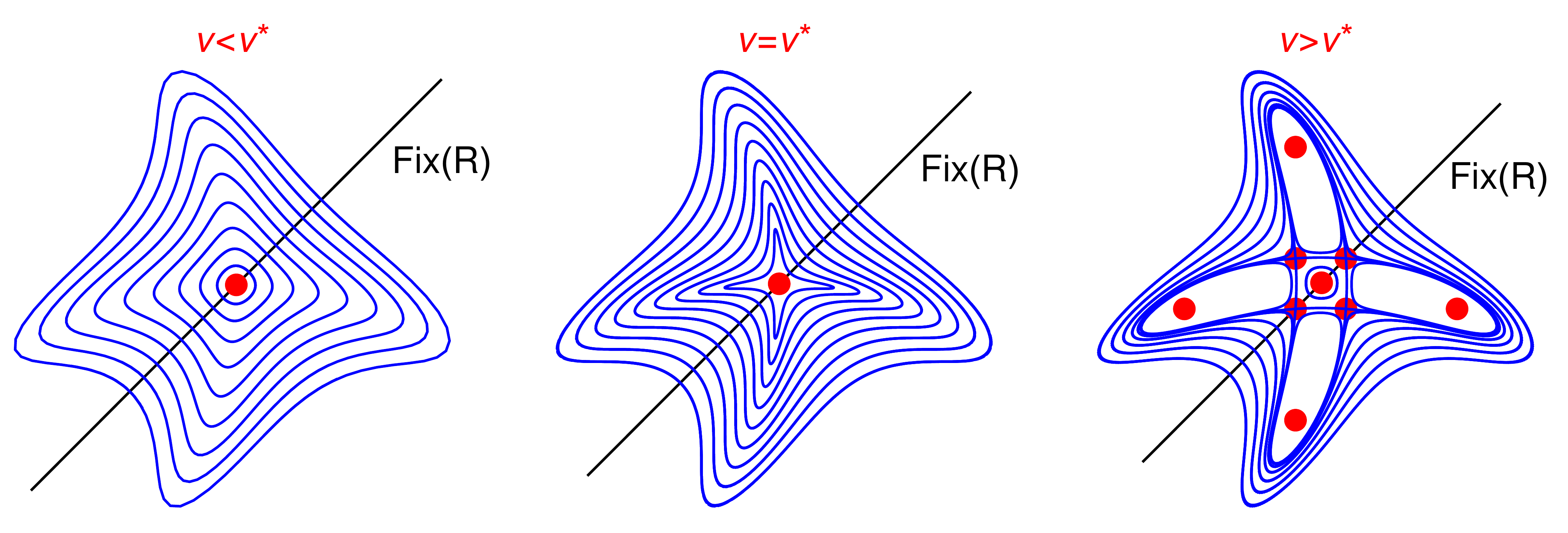}}
\centerline{\includegraphics[width=.9\textwidth]{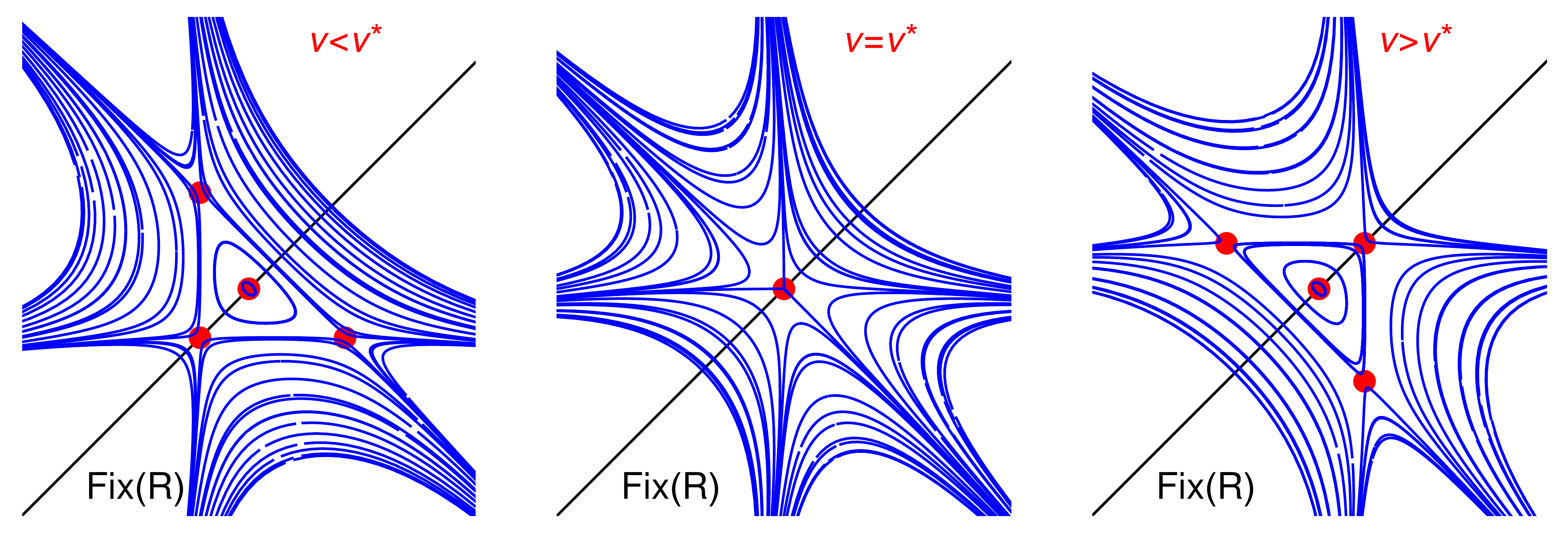}}
\caption{Upper row: period quadrupling bifurcation of a reversible planar map. After bifurcation a period-$4$ orbit is created, which intersects $\Fix(R)$ at exactly two points. Lower row: third-order \tagname bifurcation of a reversible map.}
\label{fig:resonances}
\end{figure}
This observation lead us to the alternative definitions, in which these two types of bifurcations are described by some geometric conditions on the mutual position of two intersecting curves, that solve equation $f^k_{\nu}(x)-x=0$ --- see Figure~\ref{fig:period-tupling-and-tag}. These definitions are motivated by our algorithm for validation of bifurcations (Section~\ref{sec:theory}), which is geometric in its spirit. Then, in Section~\ref{sec:unfolding} we will show, that these geometric definitions along with some non-degeneracy conditions imply standard unfolding (\ref{eq:normalFormPT})--(\ref{eq:normalFormTAG}) of these bifurcations.
\begin{figure}
\centerline{
  \includegraphics[width=.45\textwidth]{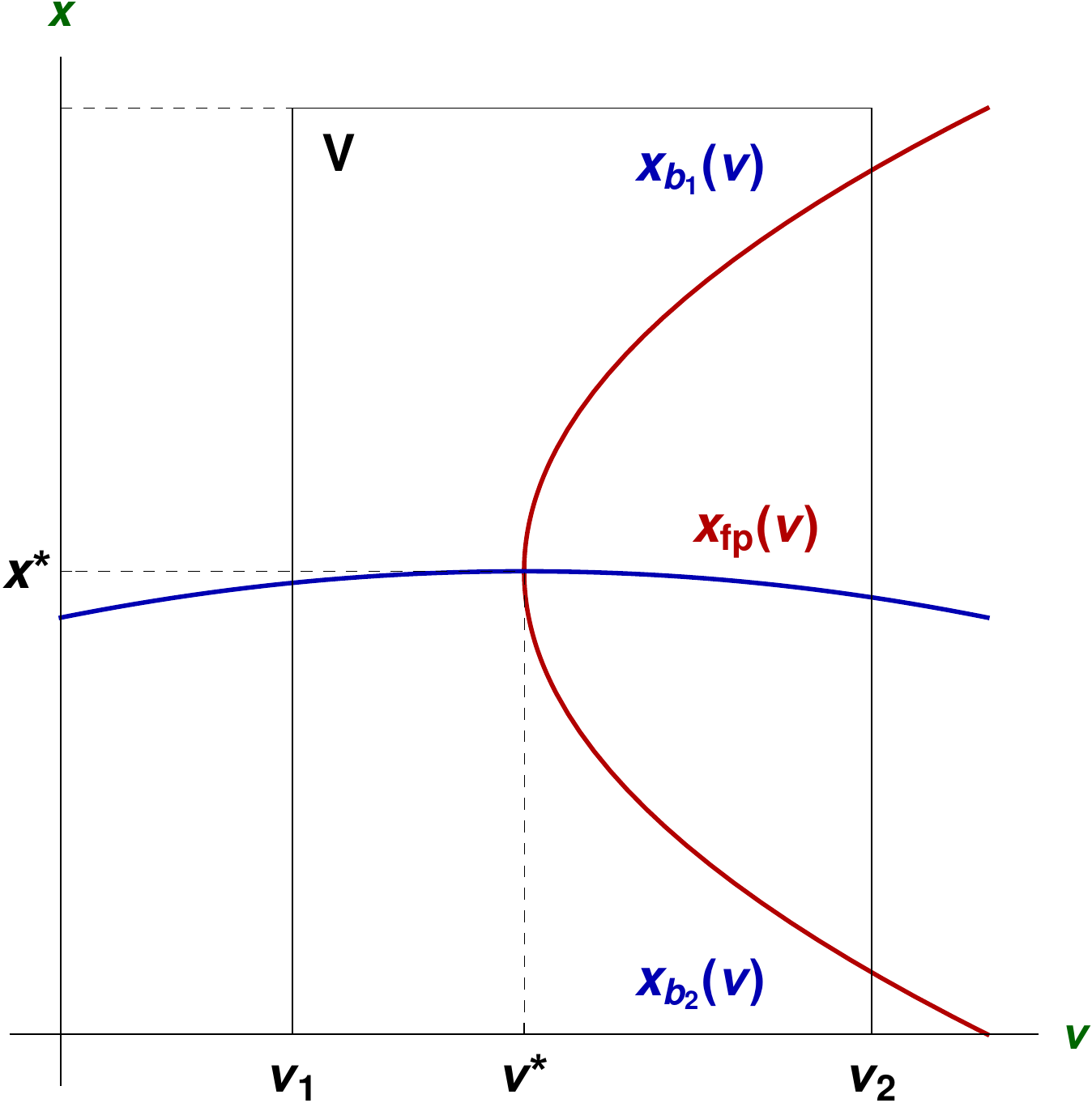}\hskip.1\textwidth
  \includegraphics[width=.45\textwidth]{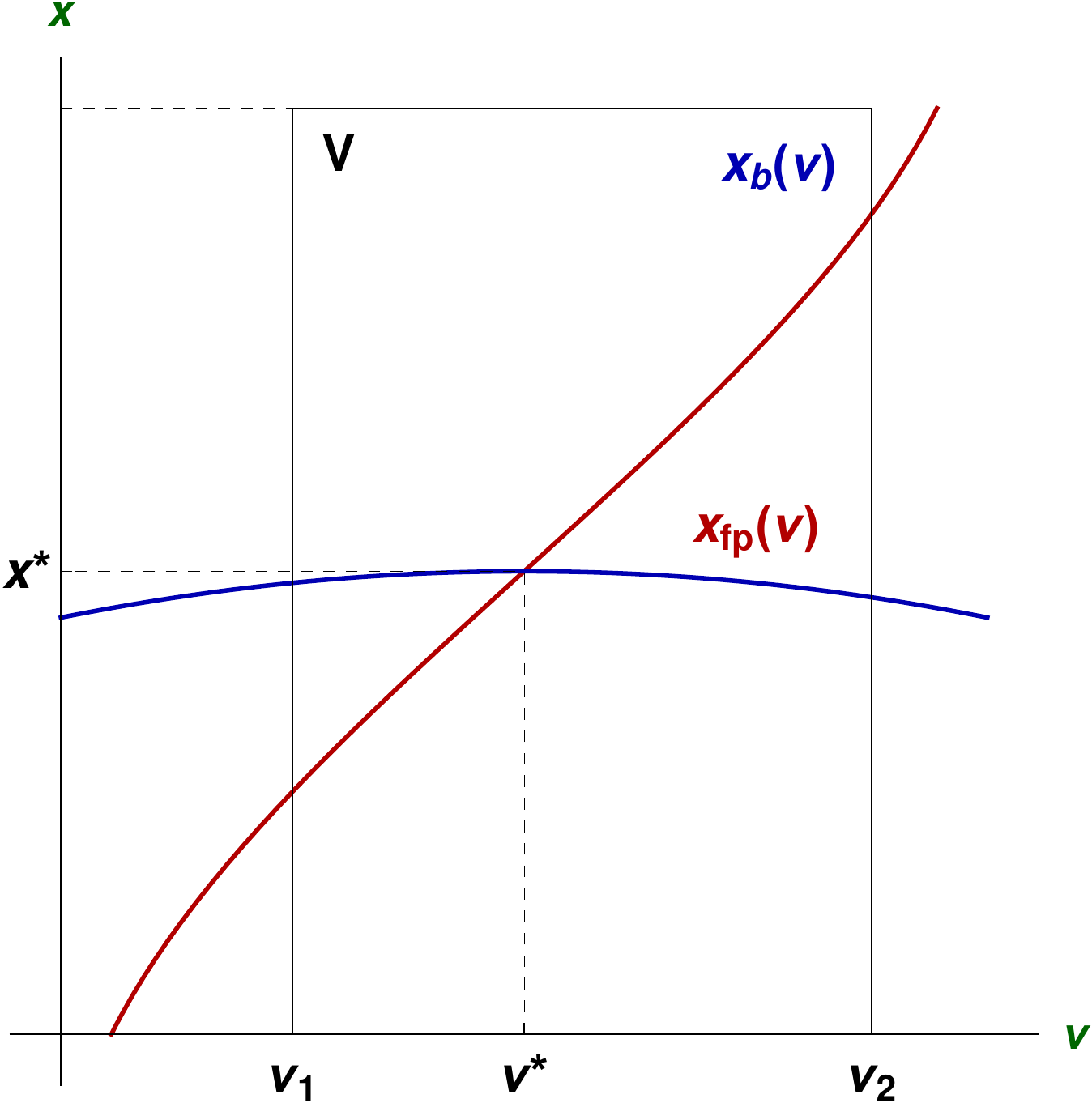}
  }
\caption{Geometry of (left) period-tupling and (right) \tagname bifurcations.}
\label{fig:period-tupling-and-tag}
\end{figure}
\begin{definition}\label{def:tupDefinition}
Let $\mathcal{C}\colon X\to \{\emph{\texttt{true},\,\texttt{false}}\}$ be a predicate and let $k$ be a positive integer. We say that $f_{\nu}\colon M_\nu\to X$ has period \emph{$k$--tupling bifurcation} at $(\nu^*,x^*)\in M$, if there exists $V=[\nu_1, \nu_2] \times U \subset M$, such that $(\nu^*,x^*) \in \mathrm{int}\,V$ and there are continuous and smooth in the interior of their domains functions
 \begin{equation*}
 \xfp \colon [\nu_1, \nu_2] \longrightarrow \mathrm{int}\,U, \qquad x_{b_1},x_{b_2} \colon [\nu^*, \nu_2] \longrightarrow \mathrm{int}\,U,
 \end{equation*} 
such that $x_{b_1}(\nu^*)=x_{b_2}(\nu^*)=\xfp(\nu^*)=x^*$ and the following conditions are satisfied.
\begin{enumerate}
\item Periodic points: \begin{equation*}
 \begin{array}{rcll}
  
  \Fix(f^p_{\nu}, U,\mathcal{C}) &=& \{ \xfp(\nu) \},&  \nu\in[\nu_1,\nu^*],\ 1\leq p\leq k, \\
  \Fix(f^p_{\nu},U,\mathcal{C})  &=& \{ \xfp(\nu)\} ,&  \nu\in[\nu^*,\nu_2],\ 1\leq p<k,\\
  \Fix(f^k_{\nu},U,\mathcal{C})  &=& \{ \xfp(\nu), x_{b_1}(\nu),x_{b_2}(\nu) \} ,& \nu\in[\nu^*,\nu_2],\\
  \#\Fix(f^k_{\nu},U,\mathcal{C}) &=&3,  &\nu \in(\nu^*,\nu_2].
  \end{array}
\end{equation*}
\item If $k$ is even, then for $\nu\in(\nu^*,\nu_2]$ there holds
\begin{equation*}
 f_{\nu}^{k/2}(x_{b_1}(\nu)) = x_{b_2}(\nu), \qquad f_{\nu}^{k/2}(x_{b_2}(\nu)) = x_{b_1}(\nu).
\end{equation*}
\end{enumerate}
\end{definition}
\begin{definition}\label{def:touchDefinition}
Let $\mathcal{C}\colon X\to \{\emph{\texttt{true},\,\texttt{false}}\}$ be a predicate and let $k$ be a positive integer. We say that $f_{\nu}\colon M_\nu\to X$ has \emph{$k^\mathrm{th}$--order \tagname bifurcation} at $(\nu^*,x^*)\in M$, if there exists $V=[\nu_1, \nu_2] \times U \subset M$, such that $(\nu^*,x^*) \in \mathrm{int}\,V$ and there are continuous and smooth in the interior of their domains functions
 \begin{equation*}
 \xfp, x_b \colon [\nu_1, \nu_2] \longrightarrow \mathrm{int}\,U,
 \end{equation*} 
satisfying the following conditions.
\begin{enumerate}
 \item Periodic points: 
 \begin{equation*}
 \begin{array}{rcll}
  \Fix(f^p_{\nu},U,\mathcal{C}) &=& \{ \xfp(\nu)\} ,&  \nu\in[\nu_1,\nu_2],\ 1\leq p<k,\\
  \Fix(f^k_{\nu},U,\mathcal{C}) &=& \{ \xfp(\nu), x_{b}(\nu) \} ,& \nu\in[\nu_1,\nu_2].
  \end{array}
\end{equation*}
\item The curves $\xfp$ and $x_b$ intersect at exactly one point $x^*=\xfp(\nu^*) = x_b(\nu^*)$. 
\end{enumerate}
\end{definition}
The above two definitions are slightly more general than the normal forms (\ref{eq:normalFormPT})--(\ref{eq:normalFormTAG}). Indeed, it is easy to see that the functions 
\begin{eqnarray*}
 f^k_\nu(x) &=& x(\nu-x^2)(\nu^{2n}+x^{2m}),\\
 f^k_\nu(x) &=& x(\nu-x)(\nu^{2n}+x^{2m})
\end{eqnarray*}
also satisfy geometric conditions from Definition~\ref{def:tupDefinition} and Definition~\ref{def:touchDefinition} for any $n,m\in\mathbb N$.

\section{Validation of bifurcations of symmetric periodic orbits: reversible maps}
\label{sec:theory}
Let $f\colon M\to X$ be a family of $R$--reversible maps satisfying our standing assumption \textbf{C1}.
In this section we propose a general framework for computer-assisted verification that a branch of $R$-symmetric period-$2$ points for $f_\nu$ undergoes \isoname, period-tupling or \tagname bifurcation. Let $V\subset M$ be an explicit neighbourhood of an approximate bifurcation point (both are input to the algorithm). The algorithm is split into two steps.

In the first step we validate, that the set of points $(\nu,x)\in V$, such that $x\in \Fix(R)$ is either period-$2k$ point or period-$2$ point for $f_\nu$ is a union of two regular curves intersecting at exactly one point. This is a common step for period-tupling and \tagname bifurcations and it will be described in Section~\ref{sec:intersection}. In the case of \isoname bifurcation, we need an additional constraint, which allows to isolate the primary curve of period-$2$ points and the branching-off curve of period-$2$ points. In Section~\ref{sec:isochronous} we will show, that such a problem can be solved, if the system has an additional symmetry, which commutes with $R$.

The second step is specific for each type of bifurcation. We provide checkable by means of rigorous numerics conditions, which guarantee, that the two curves intersect in a manner given by Definition~\ref{def:tupDefinition} and Definition~\ref{def:touchDefinition}. 

Finally, in Section~\ref{sec:unfolding} we show that some additional non-degeneracy conditions lead to standard unfolding of the computed bifurcations to their normal forms (\ref{eq:normalFormPT}) and (\ref{eq:normalFormTAG}).

\subsection{Bifurcation as two intersecting curves}\label{sec:intersection}
We are mostly interested in studying bifurcations of symmetric periodic orbits of $R$--reversible ODEs. In continuous-time $R$--reversible dynamical systems, a trajectory is $R$--symmetric, if it is either an equilibrium belonging to $\Fix(R)$ or it intersects twice the set $\Fix(R)$ --- see \cite{L}. In order to obtain isolated symmetric periodic orbits we assume, that $\mathrm{dim}\, X=2n$ and $\Fix(R)$ is an $n$-dimensional submanifold.

Let us fix $k\geq1$ and assume that $(\hat\nu, \hat x)$ is a good numerical approximation of a bifurcation point, i.e. one of the eigenvalues of $Df_{\hat\nu}(\hat x)$ is close to $1:k$ resonance. Since all types of bifurcations we are studying are local phenomena, we assume, that there is an open interval $\mathcal J$ and an open set $U\subset X$, such that $(\hat \nu, \hat x)\in\mathcal J\times U$ and for $\nu\in\mathcal J$ the function $f_\nu^{2k}$ is defined on $U$. We also assume that there are local coordinates in $U$ in which 
\begin{equation*}
\Fix(R,U)=\left\{(p,q) \in U : p,q\in\mathbb R^n,\ q=0\right\}.
\end{equation*}
We will use the same symbols $(p,q)$ as a coordinate system near $f_{\hat \nu}(\hat x)$. Solutions to 
\begin{equation}\label{eq:ImplicitEquation} 
  \pi_{q} (f_\nu^k(p,0)) =0
\end{equation}
are $R$--symmetric periodic points of $f_\nu$ with principal period not larger than $2k$. The solution set to (\ref{eq:ImplicitEquation}) near $(\hat \nu, \hat x)$ is expected to be a union of two curves, one of which corresponds to the fixed points of $f_\nu^2$ and the second corresponding to period-$2k$ points of $f_{\nu}$.

\begin{remark}
 There are rigorous numerical methods for validation, that a solution set to an implicit equation forms a regular curve over an explicit domain --- see for example \cite{BW,BLM,WZ3}. In this paper we propose slightly improved version of the Interval Newton Operator \cite{A,M,N}. Since it is only an auxiliary step of the main algorithm for validation of bifurcations, we postpone this improvement to \ref{sec:continuation}. 
\end{remark}
The case $k=1$ (\isoname bifurcation) will require additional constraint and will be discussed in Section~\ref{sec:isochronous}. If $k>1$, we may apply the method described in \ref{sec:continuation} to solve (\ref{eq:ImplicitEquation}) on an explicit range of parameter values $\nu\in\mathcal J$. In what follows we assume, that there is a smooth curve  
\begin{equation}\label{eq:fpparam}
\mathcal{J}\ni \nu \to \xfp(\nu) = (p_1(\nu),p_2(\nu),0)\in \mathbb R\times\mathbb R^{n-1}\times\mathbb R^n,
\end{equation}
such that for $\nu\in\mathcal J$ and $i=1,\ldots,k-1$
\begin{equation}\label{eq:fpEquation}
\Fix(f_\nu^{2i},U,\mathcal C_R) = \{\xfp(\nu)\}. 
\end{equation}
In most cases, the function $\xfp$ cannot be computed exactly, but using rigorous numerics we can prove, that it exists and we can find bounds for $\xfp(\nu)$ and its derivatives. Although it is not required, it is desirable for further numerical computation to choose the coordinate system in $\Fix(R,U)$ and the decomposition $p=(p_1,p_2)$ so that $p_1'(\hat \nu)\approx 0$. 

The idea of computing the second curve of $R$-symmetric, period-$2k$ points, which solves (\ref{eq:ImplicitEquation}) goes as in \cite{WZ3}. At first, we perform the Lyapunov-Schmidt reduction \cite{CH}. We split $q=(q_1, q_2) \in \mathbb R\times\mathbb R^{n-1} $ and using the method from~\ref{sec:continuation} we solve for a function $p_2 = p_2(\nu, p_1)$ satisfying implicit equation
\begin{equation}\label{eq:LSreduction} 
\pi_{q_2} (f_\nu^k(p_1,p_2,0)) = 0.
\end{equation}
Assuming that $p_2 = p_2(\nu, p_1)$ is locally unique solution to (\ref{eq:LSreduction}) in $\mathcal J\times U$, we can define so-called \emph{bifurcation function} 
\begin{equation}\label{eq:Gk}
G_k(\nu, p_1) =  \pi_{q_1} (f_\nu^k(p_1,p_2(\nu, p_1), 0)).
\end{equation}
In this way we reduced the problem of finding zeros of (\ref{eq:ImplicitEquation}) to a problem of finding zeros of a bivariate scalar-valued function (\ref{eq:Gk}). By the construction, the function $G_k$ vanishes at $(\nu, p_1(\nu))$, for $\nu \in \mathcal{J}$ --- see (\ref{eq:fpparam}). Therefore, we can factorize it in the following way  

\begin{equation}\label{eq:GkFactorization}
G_k(\nu, p_1) = ( p_1 -p_1(\nu)) g_k(\nu , p_1),
\end{equation}
where 
\begin{equation}\label{eq:gk}
g_k(\nu, p_1) = \int_0^1 \frac{\partial G_k}{\partial p_1}(\nu, t(p_1 - p_1(\nu ))+ p_1(\nu))dt.
\end{equation}
The function $g_k$ is called the \emph{reduced bifurcation function}. Recall, that in most cases the function $\nu\to p_1(\nu)$ is unknown and we have only a rigorous bound on it and its derivatives. Similarly as $p_1$, the function $g_k$ cannot be computed exactly. However, it possible to bound values and partial derivatives of $g_k$ using integral representation (\ref{eq:gk}). 

In the second step, we apply the method from \ref{sec:continuation} to prove, that the set of zeroes of $g_k$ can be parametrized as a smooth curve $p_1\to \nu(p_1)$. To sum up, in addition to \textbf{C1}, we make the following standing assumptions, which are all checkable be means of rigorous numerical methods:
\begin{itemize}
 \item[\textbf{C2:}] $P_1$ is a closed interval and $P_2$ is closed set, such that $p_1(\nu) \in \mathrm{int}\,P_1$ for $\nu \in \mathcal{J}$ and the solution set to (\ref{eq:LSreduction}) in $\mathcal J\times P_1\times P_2$ is a graph of smooth function $p_2\colon\mathcal{J}\times P_1\to \mathrm{int}\,P_2$;
 \item[\textbf{C3:}] there exists a smooth function $P_1 \ni p_1 \to \nu(p_1) \in \mathcal{J}$, such that  
$$ \left \{(\nu, p_1)\in \mathcal{J} \times P_1 :  g_k(\nu, p_1) = 0  \right \}\ = \ \left \{ (\nu(p_1), p_1 )  :  p_1 \in P_1 \right \}, $$
where $g_k$ is defined by (\ref{eq:gk});
\item [\textbf{C4:}] there holds
\begin{equation*}
0\notin \frac{\partial^2 G_k}{\partial \nu\partial p_1}(\mathcal J\times P_1) + \frac{\partial^2 G_k}{\partial p_1^2}(\mathcal J\times P_1)p_1'(\mathcal J).
\end{equation*}
\end{itemize}
We will finish this section by showing, that assumptions \textbf{C1--C3} imply, that the two curves, which solve (\ref{eq:ImplicitEquation}) intersect at some point, while \textbf{C4} implies that this intersection is unique.
\begin{lemma}\label{lem:intersection} Let $s:A\to B$ and $t:B\to A$ be continuous, where $A$ is an open interval, $B$ is a~closed interval and $s(A)\subset \mathrm{int}\,B$. Then there exists $(a^*,b^*)\in \mathrm{int}\left(A\times B\right)$, such that $(a^*,s(a^*)) = (t(b^*),b^*)$.
\end{lemma}
\begin{proof}
Since $s(A)\subset \mathrm{int}B$, the set $\left(A\times B\right)\setminus \{(a,s(a)) : a\in A\}$ has two connected components containing the points $(t(\min{B}),\min{B})$ and $(t(\max{B}),\max{B})$, respectively. Since $t$ is continuous, there is $b^*\in \mathrm{int}\,B$ such that $(t(b^*),b^*)=(a^*,s(a^*))$ for some $a^*\in A$. 
\end{proof}

\begin{lemma}\label{lem:uniqueIntersection}
Under assumptions \textbf{C1--C4}, there is a unique point $(\nu^*,p_1^*)\in\mathcal J\times P_1$, such that $\nu^*=\nu(p_1^*)$ and $p_1^*=p_1(\nu^*)$.
\end{lemma}
\begin{proof}
From \textbf{C2--C3} and Lemma~\ref{lem:intersection} the curves $\mathcal J\ni\nu\to p_1(\nu)\in P_1$ and $P_1\ni p_1\to \nu(p_1)\in\mathcal J$ intersect at some point $(\nu^*,p_1^*)$. Assume that $(\hat \nu,p_1(\hat\nu))$ is another intersection point for some $\hat\nu\in \mathcal J$. For $\tau\in[0,1]$ we define $w(\tau)=\nu^*+\tau(\hat\nu-\nu^*)$ and 
$$
u(\tau) = g(w(\tau),p_1(w(\tau)).
$$
From \textbf{C3} we have $u(0)=u(1)=0$ and $0=u(1)-u(0) = S(\hat\nu-\nu^*)$, where
\begin{equation}\label{eq:S}
S=\int_0^1\frac{\partial g}{\partial\nu}(u(\tau)) 
  + \frac{\partial g}{\partial p_1}(u(\tau))p_1'(w(\tau))d\tau.
\end{equation}
We will argue, that $S\in\frac{\partial^2 G_k}{\partial \nu\partial p_1}(\mathcal J\times P_1) + \frac{\partial^2 G_k}{\partial p_1^2}(\mathcal J\times P_1)p_1'(\mathcal J)$ and thus by \textbf{C4} there must be $\hat\nu=\nu^*$. For $(\nu,p_1)\in\mathcal J\times P_1$ we have
\begin{eqnarray*}
 \frac{\partial g}{\partial p_1}(\nu,p_1) &=& 
 \int_0^1\frac{\partial^2G}{\partial p_1^2}(\nu,t(p_1-p_1(\nu))+p_1(\nu))tdt\quad\text{and}\\
 \frac{\partial g}{\partial\nu}(\nu,p_1) &=& 
 \int_0^1\frac{\partial^2G}{\partial\nu\partial p_1}(\nu,t(p_1-p_1(\nu))+p_1(\nu))dt + \\
 && \int_0^1\frac{\partial^2G}{\partial p_1^2}(\nu,t(p_1-p_1(\nu))+p_1(\nu))p_1'(\nu)(1-t)dt.
\end{eqnarray*}
In particular, for fixed $\tau\in[0,1]$ and $(\nu,p_1)=u(\tau)$, both second partial derivatives of $G$ are evaluated at a point which does not depend on $t$. Therefore
\begin{equation}\label{eq:gp1}
 \frac{\partial g}{\partial p_1}(u(\tau)) = 
 \int_0^1\frac{\partial^2G}{\partial p_1^2}(u(\tau))tdt = \frac{1}{2}\frac{\partial^2G}{\partial p_1^2}(u(\tau)).
\end{equation}
Similarly
\begin{multline}\label{eq:gnu}
 \frac{\partial g}{\partial\nu}(u(\tau)) =
 \int_0^1\frac{\partial^2G}{\partial\nu\partial p_1}(u(\tau))
 + \frac{\partial^2G}{\partial p_1^2}(u(\tau))p_1'(w(\tau))(1-t)dt = \\
 \frac{\partial^2G}{\partial\nu\partial p_1}(u(\tau))
 + \frac{1}{2}\frac{\partial^2G}{\partial p_1^2}(u(\tau))p_1'(w(\tau)).
\end{multline}
Using (\ref{eq:S})--(\ref{eq:gnu}) we obtain 
\begin{equation*}
S = \int_0^1\frac{\partial^2G}{\partial\nu\partial p_1}(u(\tau))
 + \frac{\partial^2G}{\partial p_1^2}(u(\tau))p_1'(w(\tau))d\tau\in\frac{\partial^2 G_k}{\partial \nu\partial p_1}(\mathcal J\times P_1) + \frac{\partial^2 G_k}{\partial p_1^2}(\mathcal J\times P_1)p_1'(\mathcal J).
\end{equation*}
\end{proof}


\subsection{Validation of period-tupling and \tagname bifurcations}\label{sec:tuptag}
Assumptions \textbf{C1--C4} are easily checkable by means of rigorous numerics. The curve of $R$-symmetric period-$2k$ points is defined on an explicit domain $P_1$, which makes it possible for further continuation of this branch by the standard methods \cite{BW,BLM,GLM,WZ3} (see also ~\ref{sec:continuation}). On the other hand, Lemma~\ref{lem:uniqueIntersection} guarantees, that there is an unique bifurcation point in the domain under consideration. The type of bifurcation, however, is unknown.

In what follows, we derive conditions, which along with \textbf{C1--C4} guarantee, that period tupling or \tagname bifurcation occurs in the sense of Definition~\ref{def:tupDefinition} and Definition~\ref{def:touchDefinition}, respectively.

\begin{theorem}\label{thm:tag} Under assumptions \textbf{C1--C4}, if $k>1$ is odd and 
 \begin{equation}\label{eq:nondegeneracyConditionTAG}
0\notin\frac{\partial^2 G_k}{\partial p_1^2}(\mathcal J\times P_1)
 \end{equation}
then  $f_\nu^2$ has $k^{\mathrm{th}}$-order \tagname bifurcation at some point $(\nu^*, x^*)\in \mathcal{J} \times  P_1\times P_2\times \{0\}$.
\end{theorem}
\begin{proof}
From Lemma~\ref{lem:uniqueIntersection} there is a unique intersection point $(\nu^*,p_1^*)$ of two curves of zeroes of $G_k$ in $\mathcal J\times P_1$. 
From (\ref{eq:nondegeneracyConditionTAG}) it follows that 
\begin{equation}\label{eq:nuPrimNonzero}
\frac{\partial g_k}{\partial p_1}(\nu^*,p_1^*) = \int_0^1\frac{\partial^2 G_k}{\partial p_1^2}(\nu^*,p_1^*)tdt = 
\frac{1}{2}\frac{\partial^2 G_k}{\partial p_1^2}(\nu^*,p_1^*)\neq 0.
\end{equation}
Differentiating the identity $g_k(\nu(p_1),p_1)\equiv 0$ we obtain
$$
\frac{\partial g_k}{\partial \nu}(\nu^*,p_1^*)\nu'(p_1^*) + \frac{\partial g_k}{\partial p_1}(\nu^*,p_1^*) = 0.
$$
From the above and (\ref{eq:nuPrimNonzero}) we conclude, that $\nu'(p_1^*)\neq0$. Therefore, there is an interval $[\nu_1,\nu_2]\subset \mathcal J$, such that $\nu^*=\nu(p_1^*)\in(\nu_1,\nu_2)$ and the inverse function $\widetilde x_b = \nu^{-1}$ is defined on it. Since $p_1^*\in\mathrm{int}\, P_1$, the interval $[\nu_1,\nu_2]$ can be chosen so that $\widetilde x_b([\nu_1,\nu_2])\subset \mathrm{int}\,P_1$.

Define $U_\varepsilon = P_1\times P_2\times (-\varepsilon,\varepsilon)^n$. 
We assumed, that $\Fix(R)$ is a submanifold given locally by $q=0$. Therefore, there exists sufficiently small $\varepsilon>0$, such that for $\nu\in[\nu_1,\nu_2]$ we have 
$$\Fix(f_\nu^{2k},U_{\varepsilon},\mathcal C_R) = \Fix(f_\nu^{2k},P_1\times P_2\times\{0\},\mathcal C_R) = \{\xfp(\nu),x_b(\nu)\},$$ where 
$$
x_b(\nu) = \left(\widetilde x_b(\nu),p_2(\nu,\widetilde x_b(\nu)),0\right)\in U_{\varepsilon}.
$$
From (\ref{eq:fpEquation}) it follows that 
$$\Fix(f_\nu^{2i},U_{\varepsilon},\mathcal C_R) = \{\xfp(\nu)\}$$
for $i=1,\ldots,k-1$. Thus, all requirements from Definition~\ref{def:touchDefinition} are satisfied on the set $[\nu_1,\nu_2]\times U_{\varepsilon}$ and for the point $x^* = (p_1^*,p_2(\nu^*,p_1^*),0)$.
\end{proof}


The next theorem provides a framework for computer-assisted verification of the presence of period $k$-tupling bifurcation.
\begin{theorem}\label{thm:pt} Assume \textbf{C1--C4}. If $k\geq 2$ is even and 
 \begin{equation}\label{eq:nondegeneracyConditionPT}
\frac{\partial^3 G_k}{\partial p_1^3}(\mathcal J\times P_1) \cdot \left(\frac{\partial^2 G_k}{\partial \nu\partial p_1}(\mathcal J\times P_1) \right)^{-1} \subset \mathbb R_-,
\end{equation}
then $f_\nu^2$ has period $k$-tupling bifurcation at some point $(\nu^*, x^*)\in \mathcal{J} \times  P_1\times P_2\times \{0\}$.
\end{theorem}
\begin{proof}
From Lemma~\ref{lem:uniqueIntersection} we know, that the curves $\mathcal J\ni\nu\to p_1(\nu)\in P_1$ and $P_1\ni p_1\to \nu(p_1)\in\mathcal J$ intersect at exactly one point $(\nu^*,p_1^*)$. We will prove, that $p_1^*$ is a proper local minimum of $p_1\to \nu(p_1)$. First we will show, that $\nu'(p^*_1)=0$.

Define $x^* = (p_1^*,p_2(\nu^*,p_1^*),0)$. Let us choose a sequence $\{p_1^n\}_{n\in\mathbb N}$, such that $p_1^n\neq p_1^*$ and $\lim_{n\to\infty}p_1^n=p_1^*$. Define $\nu^n = \nu(p_1^n)$ and $x^n = (p_1^n,p_2(\nu^n,p_1^n),0)$. Since $p_1^n\neq p_1^*$ and the intersection point is unique it follows, that $x^n$ is period-$2k$ point. Since $p_1^*\in\mathrm{int}\, P_1$, $k$ is even and $x^*$ is a fixed point for $f_{\nu^*}^2$ we conclude, that for $n$ large enough the point $y^n = f_{\nu^n}^k(x^n)$ satisfies $\pi_{p_1}y^n\in\mathrm{int}\,P_1$ and
\begin{eqnarray*}
\nu(\pi_{p_1}(y^n)) &=& \nu^n,\\
\lim_{n\to\infty}\pi_{p_1}(y^n) &=& \pi_{p_1}(x^*) = p_1^*.
\end{eqnarray*}
The point $y^n$ is an $R$-symmetric, $2k$-periodic point for $f_{\nu^n}$, so it solves (\ref{eq:ImplicitEquation}). Therefore $y^n\in P_1\times P_2\times\{0\}$. Since $x^n\neq y^n$,  assumption \textbf{C2} implies, that also $p_1^n\neq \pi_{p_1}y^n$.  By the Rolle's Theorem there is a point in the interval joining $p_1^n$ and $\pi_{p_1}y^n$ at which $\nu'$ vanishes. In consequence, arbitrary close to $p_1^*$ there is a point at which $\nu'$ is zero, which by the smoothness of $\nu'(p_1)$ implies $\nu'(p_1^*)=0$.

We will show that $\nu''(p_1^*)>0$. Differentiating $g_k(\nu(p_1),p_1)\equiv 0$  we obtain 
\begin{equation*}
 \frac{\partial g_k}{\partial p_1}(\nu^*,p_1^*) = -\frac{\partial g_k}{\partial \nu}(\nu^*,p_1^*)\nu'(p_1^*) = 0.
\end{equation*}
On the other hand 
\begin{equation*}
 0 = \frac{\partial g_k}{\partial p_1}(\nu^*,p_1^*) = \int_0^1\frac{\partial^2 G_k}{\partial p_1^2}(\nu^*,p_1^*)tdt = \frac{1}{2}\frac{\partial^2 G_k}{\partial p_1^2}(\nu^*,p_1^*).
\end{equation*}
Therefore 
\begin{equation}\label{eq:dgdv}
 \frac{\partial g_k}{\partial \nu}(\nu^*,p_1^*) = \int_0^1\frac{\partial^2 G_k}{\partial p_1\partial \nu}\left(\nu^*,p_1^*\right) + \frac{\partial^2 G_k}{\partial p_1^2}\left(\nu^*,p_1^*\right)(1-t)p_1'(\nu^*)dt 
 = \frac{\partial^2 G_k}{\partial p_1\partial \nu}\left(\nu^*,p_1^*\right).
\end{equation}
We also have
\begin{equation}\label{eq:d2gdp2}
 \frac{\partial^2 g_k}{\partial p_1^2}(\nu^*,p_1^*) = \int_0^1\frac{\partial^3 G_k}{\partial p_1^3}\left(\nu^*,p_1^*\right)t^2dt = \frac{1}{3}\frac{\partial^3 G_k}{\partial p_1^3}(\nu^*,p_1^*).
\end{equation}
Differentiating twice $g_k(\nu(p_1),p_1)\equiv 0$ with respect to $p_1$, we obtain 
\begin{equation}
 \frac{\partial g_k}{\partial \nu}(\nu^*,p_1^*)\nu''(p_1^*)
 + \frac{\partial^2 g_k}{\partial \nu^2}(\nu^*,p_1^*)\left(\nu'(p_1^*)\right)^2
 + 2\frac{\partial^2 g_k}{\partial p_1\partial \nu}(\nu^*,p_1^*)\nu'(p_1^*)
 + \frac{\partial^2 g_k}{\partial p_1^2}(\nu^*,p_1^*) 
 \equiv 0.\label{eq:d2gdv2} 
\end{equation}
Using $\nu'(p_1^*)=0$ and (\ref{eq:nondegeneracyConditionPT})--(\ref{eq:d2gdv2}) we conclude, that 
\begin{equation*}
\nu''(p_1^*) = -\frac{\partial^2 g_k}{\partial p_1^2}(\nu^*,p_1^*)\left(\frac{\partial g_k}{\partial \nu}(\nu^*,p_1^*)\right)^{-1} >0.
\end{equation*}
This proves, that the function $p_1\to\nu(p_1)$ has a proper local minimum at $p_1^*$. 

Now, we will construct the functions $x_{b_{1}},x_{b_{2}}$ as required by Definition~\ref{def:tupDefinition}. Since $\nu''(p_1^*)>0$, we can shrink $P_1$ in such a way, that $p_1^*\in\mathrm{int}\,P_1$, $\nu(\min P_1)=\nu(\max P_1)$ and $\nu$ is strictly convex on $P_1$. Given that $p_1(\nu)$ is smooth (so it has bounded slope near $\nu^*$) and $\nu'(p_1^*)=0$ we conclude, that for $\nu\in[\nu^*,\nu(\min P_1)]$ there holds $p_1(\nu)\in \mathrm{int}\,P_1$. Shrinking $\mathcal J$ from below, if necessary, we may also assume that $p_1(\nu)\in \mathrm{int}\,P_1$ for $\nu<\nu^*$ and $\nu\in\mathcal J$. 

Let us fix $\nu_2\in(\nu^*,\nu(\min P_1))$. The function $\nu(p_1)$ is monotone on $[\min P_1,p_1^*]$ and $[p_1^*,\max P_1]$. Therefore, for $\widetilde \nu\in[\nu^*,\nu_2]$ 
\begin{equation*}
 \{p_1 : \nu(p_1) = \widetilde \nu \}\ =\ \{\widetilde x_{b_1}(\widetilde \nu),\widetilde x_{b_2}( \widetilde \nu)\}
\end{equation*}
and the functions $\widetilde x_{b_i}$ can be chosen to be continuous and smooth in $(\nu^*,\nu_2)$. Define $U_\varepsilon=P_1\times P_2\times (-\varepsilon,\varepsilon)^n$. For some sufficiently small $\varepsilon>0$ the functions 
\begin{equation*}
 x_{b_1},x_{b_2}:[\nu^*,\nu_2]\to \mathrm{int}U_{\varepsilon}
\end{equation*}
given by 
\begin{equation*}
 x_{b_i}(\nu) = \left(\widetilde x_{b_i}(\nu),p_2(\nu,\widetilde x_{b_i}(\nu)),0\right).
\end{equation*}
satisfy conditions from Definition~\ref{def:tupDefinition}. From (\ref{eq:fpEquation}) it follows that 
$$\Fix(f_\nu^{2i},U_{\varepsilon},\mathcal C_R) = \{\xfp(\nu)\}$$
for $i=1,\ldots,k-1$, which completes the proof.
\end{proof}

\subsection{\Isoname bifurcations}\label{sec:isochronous}
A special type of bifurcation covered by Definition~\ref{def:tupDefinition} is when $k=1$. In the case of reversible maps, we are looking for the set of fixed points of $f_\nu^2$, which near the bifurcation point is expected to be a union of two intersecting curves. Thus, the principal branch of fixed points of $f^2_\nu$ cannot be isolated by applying Interval Newton Operator (\ref{sec:continuation}) to (\ref{eq:ImplicitEquation}). Therefore, we need an extra constraint in order to isolate two curves that intersect at the bifurcation point. A possible scenario for this kind of bifurcation is breaking some symmetry $S$, which provides desired additional constraint $\mathcal{C}_S$, as defined by (\ref{eq:predicateSymmetry}).

Let $f\colon M\subset \mathcal{J}\times X\to X$ be a family of $R$--reversible and $S$--symmetric maps. As in Section~\ref{sec:intersection}, we focus on computation of the set of $R$--symmetric fixed points for $f^2_{\nu}$. 

Let $(\hat\nu,\hat x)\in M$ be an apparent bifurcation point and let $U\subset X$ be an open set, such that $\hat x\in U$ and $\mathcal J\times U\subset M$. We assume, that there are local coordinates in $U$, such that
\begin{equation*}
\Fix(R,U)=\left\{(p,q) \in U : p,q\in\mathbb R^n,\ q=0\right\}.
\end{equation*}
We impose that $\Fix(R,U)$ and $\Fix(S,U)$ intersect transversally. In the above settings, we expect that for $\nu\in\mathcal J$ the set of double-symmetric fixed points of $f_\nu^2$ in $U$ is a single point
\begin{equation*}
\Fix(f_\nu^2,U,\mathcal C_R\wedge \mathcal C_S) = \{x_{fp}(\nu)\}.
\end{equation*}
Thus, using the method described in~\ref{sec:continuation}, it is possible to isolate the curve 
\begin{equation}\label{eq:fpCurveIso}
\mathcal{J}\ni \nu \to \xfp(\nu) = (p_1(\nu),p_2(\nu),0)\in \mathbb R\times\mathbb R^{n-1}\times\mathbb R^n
\end{equation}
from the branching-off curves $x_{b_{1,2}}(\nu)\in\Fix(R,U)$, which intersect $\xfp(\nu)$ at exactly one point located in $\Fix(R,U)\cap \Fix(S,U)$.

The remaining steps in validation of the existence of a \isoname bifurcation go as described in Section~\ref{sec:intersection} and Section~\ref{sec:tuptag}. We perform the Lyapunov-Schmidt reduction (\ref{eq:LSreduction}) and we define the bifurcation functions $G_1$ and $g_1$ by (\ref{eq:Gk}) and (\ref{eq:gk}), respectively.

In Theorem~\ref{thm:pt} we assumed that the period of branching-off orbits is even. This lead us to a~conclusion, that after half of iterations, the points sufficiently close to the bifurcation point come back to its vicinity. This allowed us to prove, that the reduced bifurcation function $g_k$ has a local minimum at the bifurcation point, which was the crucial step in the proof of the presence of period-tupling bifurcation. 

In the case of the \isoname bifurcation we cannot use this argument, because $k=1$ is odd. The next theorem says, that commutativity of $R$ and $S$ yields to the same conclusion.

\begin{theorem}\label{thm:iso} 
Let $S$ and $R$ be a symmetry and a reversing symmetry for $f_\nu$, respectively. Assume that \textbf{C1--C4} are satisfied for $k=1$ and with $\xfp$ given by (\ref{eq:fpCurveIso}). If the symmetries $R$ and $S$ commute and (\ref{eq:nondegeneracyConditionPT}) is satisfied, then $f_\nu^2$ has \isoname bifurcation at some point $(\nu^*, x^*)\in \mathcal{J} \times  P_1\times P_2\times\{0\}$.
\end{theorem}
\begin{proof}
From Lemma~\ref{lem:uniqueIntersection} there is an unique $(\nu^*,p_1^*)\in\mathrm{int}\,\mathcal J\times P_1$, such that $p_1(\nu^*)=p_1^*$ and $\nu(p_1^*)=\nu^*$. Let us set $x^*:=(p_1(\nu^*),p_2(\nu^*),0)\in \Fix(S)\cap\Fix(R)$. 

We will show, that $\nu'(p_1^*)=0$. Take a sequence $\left\{p_1^n\right\}_{n\in\mathbb N}$ of points from $P_1$, which converges to $p_1^*$ and such that $p_1^n\neq p_1^*$ for $n\in\mathbb N$. Let us define $\nu^n = \nu(p_1^n)$, $x^n = (p_1^n,p_2(\nu^n,p_1^n),0)\in\Fix(R)$ and $y^n=S(x^n)$. By the commutativity of $R$ and $S$ we have
$$
R(y^n) = R(S(x^n)) = S(R(x^n)) = S(x^n) = y^n,
$$
hence $y^n\in\Fix(R)$. Similarly $f_{\nu^n}(y^n)\in \Fix(R)$, because
\begin{equation*}
R(f_{\nu^n}(y^n)) = 
R(f_{\nu^n}(S(x^n))) = 
R(S(f_{\nu^n}(x^n))) = 
S(R(f_{\nu^n}(x^n))) = 
S(f_{\nu^n}(x^n)) = f_{\nu^n}(y^n).
\end{equation*}
This shows, that $y^n$ solves (\ref{eq:ImplicitEquation}) with $\nu=\nu^n$  and $\nu(\pi_{p_1}(y^n)) = \nu^n$. Since 
\begin{equation}\label{eq:yconvtop}
\lim_{n\to\infty}\pi_{p_1}(y^n) = \pi_{p_1}(S(x^*)) = p_1^*
\end{equation}
we conclude, that $y^n\in P_1\times P_2\times \{0\}$ for $n$ large enough.

Since $p_1^n\neq p_1^*$ and the intersection point is unique, it follows that $x^n\notin \Fix(S)$, so $x^n\neq y^n$. This and \textbf{C2} imply that also $p_1^n\neq \pi_{p_1}y^n$. By the Rolle's Theorem there is a point in the interval joining $p_1^n$ and $\pi_{p_1}y^n$ at which $\nu'$ vanishes. Given that both $p_1^n$ and $\pi_{p_1}(y^n)$ (see (\ref{eq:yconvtop})) converge to $p_1^*$ we conclude, that arbitrarily close to $p_1^*$ there is a point at which $\nu'$ is zero and thus $\nu'(p_1^*) = 0$. 

The remaining part of the proof goes as in Theorem~\ref{thm:pt}. Condition (\ref{eq:nondegeneracyConditionPT}) implies, that $\nu''(p_1^*)>0$, which proves that $p_1^*$ is a proper minimum of $\nu(p_1)$. This makes it possible to define two branches $x_{b_1}$ and $x_{b_2}$, as required in Definition~\ref{def:tupDefinition}.
\end{proof}

\subsection{General unfolding of period-tupling and \tagname bifurcations}
\label{sec:unfolding}
Theorem~\ref{thm:tag} and Theorem~\ref{thm:pt} provide frameworks for computer-assisted verification, that the two types of bifurcations occurs in the sense of geometric conditions given by Definition~\ref{def:touchDefinition} and Definition~\ref{def:tupDefinition}, respectively. In this section we will show, that the same geometric assumptions lead to standard unfolding of these bifurcations.

\begin{theorem}\label{thm:pt_normal_form}
 Under assumptions of Theorem~\ref{thm:pt} (\textbf{C1--C4} and (\ref{eq:nondegeneracyConditionPT})), there is a smooth $\nu$-dependent substitution $p_1=p_1(\nu,w)$, which brings the bifurcation function (\ref{eq:Gk}) to the normal form 
 \begin{equation}\label{eq:pt_normal_form}
 G_k(\nu,w)=\alpha w(\nu-\beta w^2) + O(|w|^4),
 \end{equation}
 where $\alpha=\alpha(\nu)$ does not vanish at $\nu=0$ and $\beta=\beta(\nu)$ is positive at $\nu=0$.
\end{theorem}
\begin{proof}
Let $\nu^*$ be the parameter value at which bifurcation occurs. Without loosing generality we may assume, that $\nu^*=0$. From the proof of Theorem~\ref{thm:pt} we have that $\nu'(p_1^*)=0$. Differentiation of the identity $g_k(\nu(p_1),p_1)\equiv 0$ gives
\begin{equation}\label{eq:dgp1_vanish}
 \frac{\partial g_k}{\partial p_1}(\nu^*,p_1^*) = -\frac{\partial{g_k}}{\partial \nu}(\nu^*,p_1^*)\nu'(p_1^*) =0.
\end{equation}
First, a $\nu$-dependent substitution $p_1=z + p_1(\nu)$ brings $G_k$ to the form 
\begin{equation}\label{eq:Gzfactorization}
G_k(\nu,z) = zg_k(\nu,z) = zg_k(\nu,z+p_1(\nu)).
\end{equation}
Put $z^*=0$. Differentiation of the product $G_k(\nu,z)=zg_k(\nu,z)$ gives 
\begin{eqnarray}\label{eq:Gk_derivatives1}
 \begin{array}{rcl}
 \frac{\partial G_k}{\partial z}(\nu^*,z^*) &=& g_k(\nu^*,z^*) + z^*\frac{\partial g_k}{\partial z}(\nu^*,z^*) = 0,\\
 \frac{\partial^2 G_k}{\partial z^2}(\nu^*,z^*) &=& 2\frac{\partial g_k}{\partial z}(\nu^*,z^*) + z^*\frac{\partial^2g_k}{\partial z^2}(\nu^*,z^*) = 0,
 \end{array}
\end{eqnarray}
because $g_k(\nu^*,z^*)=0$ and from (\ref{eq:dgp1_vanish}) also $\frac{\partial g_k}{\partial z}(\nu^*,z^*)=\frac{\partial g_k}{\partial p_1}(\nu^*,p_1^*)=0$. The non-degeneracy condition (\ref{eq:nondegeneracyConditionPT}) and \textbf{C4} imply that 
\begin{equation}\label{eq:Gk_derivatives2}
 \begin{array}{rcl}
\frac{\partial^3 G_k}{\partial z^3}(\nu^*,z^*) &=& \frac{\partial^3 G_k}{\partial p_1^3}(\nu^*,p_1^*)\neq0,\\
\frac{\partial^2 G_k}{\partial z\partial \nu}(\nu^*,z^*) &=& \frac{\partial^2 G_k}{\partial p_1\partial \nu}(\nu^*,p_1^*) + \frac{\partial^2 G_k}{\partial p_1^2}(\nu^*,p_1^*)p_1'(\nu^*) \neq 0.
 \end{array}
\end{equation}
Using (\ref{eq:Gzfactorization})--(\ref{eq:Gk_derivatives2}) we can write $G_k$ in the form
\begin{equation*}
 G_k(\nu,z) = c_1\nu z + c_2\nu z^2 + c_3z^3 + O(|z|^4),
\end{equation*}
where $c_i$, $i=1,2,3$ are $\nu$-dependent coefficients with $c_1(0)\neq0$ and $c_3(0)\neq 0$. Consider a~$\nu$-dependent substitution of the form
$$
z = w - h(\nu)w^2.
$$
Then 
$$G_k(\nu,w) = c_1\nu w + \left(-c_1 h + c_2\right)\nu w^2 + (c_3-2c_2\nu h)w^3 + O(|w|^4).$$
In order to annihilate the term $\nu w^2$ we should take $h=c_2/c_1$, which is well defined near $\nu^*=0$, because $c_1(0)\neq0$. With such choice of $h$, we have 
$$G_k(\nu,w) = c_1\nu w + d_3w^3 + O(|w|^4),$$
where $d_3=c_3-2c_2^2\nu/c_1$ and $d_3(0)=c_3(0)\neq 0$. Put $\beta = -d_3/c_1$. From (\ref{eq:nondegeneracyConditionPT}) and (\ref{eq:Gk_derivatives2}) we have $c_1(0)c_3(0)<0$. Therefore $\beta(0) = -c_3(0)/c_1(0)>0$. Setting $\alpha=c_1$ we conclude, that $G_k$ takes the required form (\ref{eq:pt_normal_form}).
\end{proof}

\begin{theorem}
 Under assumptions of Theorem~\ref{thm:tag} (\textbf{C1--C4} and (\ref{eq:nondegeneracyConditionTAG})), there is a smooth $\nu$-dependent substitution $p_1=p_1(\nu,w)$ which brings the bifurcation function (\ref{eq:Gk}) to the normal form 
 \begin{equation}\label{eq:tag_normal_form}
 G_k(\nu,w)=\alpha w(\nu-\beta w) + O(|w|^3),
 \end{equation}
 where $\alpha=\alpha(\nu)$, $\beta=\beta(\nu)$ and they are non-zero at $\nu=0$.
\end{theorem}
\begin{proof}
Let $\nu^*$ be the parameter value at which bifurcation occurs. Without loosing generality we may assume that $\nu^*=0$. Reasoning as in the proof of Theorem~\ref{thm:pt_normal_form}, substitution $z=p_1-p_1(\nu)$ brings $G_k$ to the form $G_k(\nu,z)=zg_k(\nu,z)$. Put $z^*=0$. From \textbf{C1--C4} and (\ref{eq:nondegeneracyConditionTAG}) we have that 
\begin{eqnarray*}
 \begin{array}{rcl}
 \frac{\partial G_k}{\partial z}(\nu^*,z^*) &=& g_k(\nu^*,z^*) + z^*\frac{\partial g_k}{\partial z}(\nu^*,z^*) = 0,\\
 \frac{\partial^2 G_k}{\partial z^2}(\nu^*,z^*) &=& \frac{\partial^2 G_k}{\partial p_1^2}(\nu^*,p_1^*) \neq 0,\\
\frac{\partial^2 G_k}{\partial z\partial \nu}(\nu^*,z^*) &=& \frac{\partial^2 G_k}{\partial p_1\partial \nu}(\nu^*,p_1^*) + \frac{\partial^2 G_k}{\partial p_1^2}(\nu^*,p_1^*)p_1'(\nu^*) \neq 0.
 \end{array}
\end{eqnarray*}
Hence, in these coordinates we have
\begin{equation*}
 G_k(\nu,z) = c_1\nu z + c_2 z^2 + O(|z|^3),
\end{equation*}
with $c_1(0)\neq0$ and $c_2(0)\neq 0$. Setting $\alpha = c_1$ and $\beta = c_2/c_1$ we obtain the required normal form (\ref{eq:tag_normal_form}).
\end{proof}

\section{Validation of bifurcations of symmetric periodic orbits: reversible Hamiltonian systems}\label{sec:hamiltonian}
In this section we show, how to adopt the general framework described in Section~\ref{sec:theory} to autonomous Hamiltonian systems. We consider an $R$--reversible Hamiltonian system
\begin{equation}\label{eq:genHamiltonian}
\dot x = J\nabla H(x), 
\end{equation}
where $J$ is the standard symplectic matrix and $H \colon \mathbb{R}^{2n + 2} \to \mathbb{R}$ is $\mathcal{C}^4$-smooth. First, we choose a $\mathcal C^3$-smooth Poincar\'e section $\Pi\subset \mathbb{R}^{2n + 2}$. If $\Fix(R) \subset \Pi$, by \cite[Lemma 3.3]{W} the Poincar\'e map $\PM\colon\Pi\to\Pi$ is $R|_\Pi$--reversible, too. In what follows we will study bifurcations of $R$-symmetric fixed points for $\PM^2$. 

In autonomous Hamiltonian systems, a natural choice of the bifurcation parameter is the value of $H$, because it is a constant of motion. On the other hand, from the point of view of rigorous numerics, it is much easier and efficient to work in the phase space coordinates and parametrize periodic orbits for $\PM$ by one of them. 

This dissonance between numerical efficiency and formal description of a bifurcation is solved in the following way. We expect that two families of periodic points for $\PM$ intersect at a bifurcation point. This geometric condition can be checked in the phase-space coordinates. Additional conditions on $H$, which will be given in this section, will guarantee, that $H$ can be used (locally) as a bifurcation parameter.

Let us give a  brief overview of the construction we are going to perform. Period-tupling and \tagname bifurcations are local phenomena. Thus, we can choose local coordinates 
$(p_0, p,q_0,q) \in \mathbb{R} \times \mathbb{R}^{n} \times \mathbb{R} \times  \mathbb{R}^{n} $ near an apparent bifurcation point, such that 
\begin{enumerate}
 \item $\Pi =  \left \{ (p_0, p,q_0, q)  \in \mathbb{R} \times \mathbb{R}^{n} \times \mathbb{R} \times  \mathbb{R}^{n} : q_0 = 0  \right \}$ and 
 \item $\Fix(R) =\{(p_0,p,q_0,q)\in\mathbb{R} \times \mathbb{R}^{n} \times \mathbb{R} \times  \mathbb{R}^{n} : q_0=0, q=0\}\subset \Pi$.
\end{enumerate}
In order to simplify further notation we will use $(p_0,p,q)$ coordinates in $\Pi$ and we will always skip $q_0=0$ as an argument of $\mathcal P$ and $H$. We also assume, that the vector field (\ref{eq:genHamiltonian}) is transverse to $\Pi$ near an apparent bifurcation point, so that the Poincar\'e map is well defined and smooth. Define $\Pi_h = \left\{(p_0,p,q)\in\Pi :H(p_0,p,q)=h\right\}$. 

Assuming that near an apparent bifurcation point there holds
\begin{equation*}
\frac{\partial H}{\partial p_0}(p_0, p,q) \neq 0
\end{equation*}
we can conclude that the projection $\Pi_h\ni(p_0,p,q)\to (p,q)$ is a local diffeomorphism. This allows us to parametrize $\Pi_h$ locally by $(p_0(h,p,q),p,q)$ for $h$ belonging to some interval $\mathcal H$. Our goal is to specify conditions on $\PM$ and $H$, which will guarantee, that the map
\begin{equation}\label{eq:fh_family}
 f(h,p ,q) = \pi_{(p,q)}\PM(p_0(h,p,q),p,q)
\end{equation}
is well defined on some domain and it has period-tupling and/or \tagname bifurcations in the sense of Definition~\ref{def:tupDefinition} and Definition~\ref{def:touchDefinition}, respectively. 

Let us fix $k>1$. Following Section \ref{sec:theory}, we split $p=(p_1,p_2)$ and $q=(q_1,q_2)$. We assume, that the set of $R$--symmetric fixed points of $\PM^2$ near an apparent bifurcation point forms a regular curve, which is parametrized by the coordinate $p_0$ 
\begin{equation}\label{eq:fpparamH}
P_0 \ni p_0 \to  \xfp^H(p_0) := (p_0, p_1(p_0),p_2(p_0),0)\in \mathbb{R} \times  \mathbb{R} \times \mathbb R^{n-1}\times\mathbb R^n
\end{equation}
and defined on an explicit, open interval $P_0$. We also assume, that for $i=1,\ldots,k-1$ there holds
\begin{equation}\label{eq:curveFPH}
 \Fix(\PM^{2i},P_0\times P_1\times P_2\times\{0\}))=\{\xfp^H(p_0) : p_0\in P_0\}.
\end{equation}
First we perform the Lyapunov-Schmidt reduction. We assume that there is a set $P_0\times P_1\times P_2$ and a smooth function $p_2^H:P_0\times P_1\to P_2$, such that for $(p_0,p_1,p_2)\in P_0\times P_1\times P_2$ 
\begin{equation}\label{eq:LSreductionH}
\pi_{q_2} (\PM^k(p_0, p_1,p_2,0)) = 0\quad \Longleftrightarrow \quad p_2=p_2^H(p_0,p_1).
\end{equation}
Using this implicit function we can define the bifurcation function by 
\begin{equation}\label{eq:GkH}
G_k^H(p_0, p_1) =  \pi_{q_1} (\PM^k(p_0, p_1,p_2^H(p_0, p_1), 0)) 
\end{equation}
and factorize it as $G_k^H(p_0,p_1)=(p_1-p_1(p_0))g_k^H(p_0,p_1)$, where the reduced bifurcation function reads
\begin{equation}\label{eq:gkH}
g_k^H(p_0, p_1) =  \int_0^1 \frac{\partial G_k^H}{\partial p_1}\left(p_0,p_1(p_0)+t(p_1-p_1(p_0))\right)dt. 
\end{equation}
Now we can specify the standing assumptions in the context of Hamiltonian systems:
\begin{itemize}
 \item[\textbf{HC2:}] $P_1$ is a closed interval, and $P_2$ is a closed set, such that $p_1(p_0)\in \mathrm{int}\,P_1$ for $p_0\in P_0$, where $p_1(p_0)$ is given by (\ref{eq:fpparamH}) and there is a smooth function  $p_2^H:P_0\times P_1\to \mathrm{int}\,P_2$ solving (\ref{eq:LSreductionH});
 \item[\textbf{HC3:}] there exists a smooth function $P_1 \ni p_1 \to p_0(p_1) \in P_0$ such that  
$$ \left \{(p_0, p_1)\in P_0 \times P_1 :  g_k^H(p_0, p_1) = 0  \right \}\ =\ \left \{ (p_0(p_1), p_1 )  :  p_1 \in P_1 \right \}, $$
where $g_k^H$ is defined by (\ref{eq:gkH});
 \item[\textbf{HC4:}] there holds
 \begin{equation*}
 0\notin \frac{\partial^2 G_k^H}{\partial p_0\partial p_1}(P_0\times P_1) + \frac{\partial^2 G_k^H}{\partial p_1^2}(P_0\times P_1)p_1'(P_0).  
\end{equation*}
\end{itemize}
Assumptions \textbf{HC2--HC4} and Lemma~\ref{lem:uniqueIntersection} imply, that the set of $R$--symmetric fixed points for $\PM^{2k}$ in $P_0\times P_1\times P_2\times \{0\}$ is the union of two regular curves, which intersect at unique point $x^*=(p_0^*,p_1^*,p_2^H(p_0^*,p_1^*),0)$, where $p_2^H$ is defined by (\ref{eq:LSreductionH}). These parametric curves are defined on explicit range $p_0\in P_0$ and $p_1\in P_1$, respectively, which makes it possible for further continuation of these branches by the method described in \ref{sec:continuation}. On the other hand, we have no information about the type of bifurcation at the intersection point. 

Generically we expect, that the value of $H$ can be used as the bifurcation parameter. It may happen, however, that $H$ is not monotone (or even constant) along one or both curves of periodic points. In the remaining part of this section we derive (easily checkable by means of rigorous numerics) conditions, which guarantee, that 
\begin{enumerate}
\item $H$ can be used as the parameter near $x^*$ and 
\item the mapping (\ref{eq:fh_family}) undergoes one of the bifurcations defined by Definition~\ref{def:tupDefinition} and/or Definition~\ref{def:touchDefinition}.
\end{enumerate}
For further use we define four functions
\begin{eqnarray*}
 s(p_0,p,q) &=& (H(p_0,p,q),p,q),\\
 x_b^H(p_1) &=& \left(p_0(p_1),p_1,p_2^H(p_0(p_1),p_1),0\right),\\
 h_{fp}(p_0) &=& H(\xfp^H(p_0))\quad\text{and}\\
 h_b(p_1) &=& H(x_b^H(p_1)).
\end{eqnarray*}

\begin{lemma}\label{lem:substitutionS}
If $\PM^2(x^*)=x^*$ and  
\begin{equation}
\frac{\partial H}{\partial p_0}(x^*)\neq 0,\label{eq:energyP0Monotone}
\end{equation}
then there is an open interval $\mathcal H$ and an open set $U$, such that 
\begin{enumerate}
\item $s(x^*)\in \mathcal H\times U$,
\item $s^{-1}|_{\mathcal H\times U}$ is well defined diffeomorphism onto image and
\item $\PM^k\circ s^{-1}$ is defined on $\mathcal H\times U$.
\end{enumerate}
\end{lemma}
\begin{proof}
From (\ref{eq:energyP0Monotone}) it follows, that $\det \left(D s(x^*)\right) = \frac{\partial H}{\partial p_0}(x^*)\neq 0$ and thus $s$ is a local diffeomorphism near $x^*$. The inverse function takes the form
\begin{equation*}
s^{-1}(h,p,q) = (p_0(h,p,q),p,q).
\end{equation*}
Since $x^*$ is a fixed point of $\PM^2$, we can also find open sets $\mathcal H$ and $U$, such that $s(x^*)\in \mathcal H\times U$ and such that $\PM^k\circ s^{-1}$ is defined on $\mathcal H\times U$.
\end{proof}

\begin{theorem}\label{thm:Htag}
Assume \textbf{HC2--HC4} and let $x^*=(p_0^*,p_1^*,p_2^H(p_0^*,p_1^*),0)$ be the unique intersection point from Lemma~\ref{lem:uniqueIntersection}. If $k>1$ is odd, $\frac{\partial H}{\partial p_0}(x^*)\neq 0$, $h_{fp}'(p_0^*)\neq 0$ and $h_b'(p_1^*)\neq 0$, then $f^2_h$ defined by (\ref{eq:fh_family}) has $k^{\mathrm{th}}$-order \tagname bifurcation at $s(x^*)$.
\end{theorem}
\begin{proof}
Take $\mathcal H\times U$ from Lemma~\ref{lem:substitutionS}. From (\ref{eq:fpparamH})--(\ref{eq:curveFPH}) the point $x_b^H(p_1)$, $p_1\in P_1$ is of principal period $2k$ for $\PM$, except the intersection $p_1^*$. Since  $h_{fp}'(p_0^*)\neq0$ and  $h_b'(p_1^*)\neq0$, we can choose $[h_1,h_2]\subset \mathcal H$, $h^*\in(h_1,h_2)$, such that $h_{fp}^{-1}$ and $h_b^{-1}$ are defined on $[h_1,h_2]$ and both functions  
\begin{equation*}
 \xfp = \pi_{(p,q)}\circ \xfp^H\circ h_{fp}^{-1}\quad \text{and}\quad
 x_b = \pi_{(p,q)}\circ x_b^H\circ h_b^{-1}
\end{equation*}
have range in $U$. By the construction they satisfy conditions from Definition~\ref{def:touchDefinition} for the function $f$ defined by (\ref{eq:fh_family}).
\end{proof}

\begin{theorem}\label{thm:Htupling}
Assume \textbf{HC2--HC4} and let $x^*=(p_0^*,p_1^*,p_2^H(p_0^*,p_1^*),0)$ be the unique intersection point from Lemma~\ref{lem:uniqueIntersection}. If $k\geq2$ is even, $\frac{\partial H}{\partial p_0}(x^*)\neq 0$, $h_{fp}'(p_0^*)> 0$ and $h_b''(p_1^*)> 0$, then $f^2_h$ defined by (\ref{eq:fh_family}) has period $k$-tupling bifurcation at $s(x^*)$.
\end{theorem}
\begin{proof}
Take $\mathcal H\times U$ from Lemma~\ref{lem:substitutionS}. We will show that $h_b'(p_1^*)=0$. Let us fix a sequence $p_1^n\to p_1^*$, such that $p_1^*\neq p_1^n$ and $x_b^H(p_1^n)$ is defined for all $n\in \mathbb N$. Let us set $x^n=x_b^H(p_1^n)$ and $y^n=\PM^k(x^n)$. We also have $y^n\neq x^n$, because by (\ref{eq:curveFPH}) the principal period of $x^n$ is $2k$. This and \textbf{HC3} imply, that $p_1^n\neq \pi_{p_1}y^n$. Since $k$ is even and $x^*$ is a fixed point for $\PM^2$, we have that $\lim_{n\to\infty}y^n=x^*$, $y^n\neq x^n$ and $H(y^n)=H(x^n)=h_b(p_1^n)$. Hence, in the interval joining $p_1^n$ and $\pi_{p_1}y^n$ there is a point at which $h_b'$ vanishes and in consequence $h_b'(p_1^*)=0$. The assumption $h_b''(p_1^*)>0$ implies, that $p_1^*$ is a~proper local minimum of $h_b$. 

The remaining part of the construction goes as in Theorem~\ref{thm:pt}. The function $h_b$ is convex near $p_1^*\in\mathrm{int}\,P_1$ which allows us to define two continuous branches 
$$
h_{b_i}:[h^*,h_2]\to \mathrm{int}\, P_1,\quad i=1,2,
$$
of $h_b^{-1}$, for some $h_2\in\mathcal H$. The function $h_{fp}$ is invertible near $p_0^*$, because we assumed that $h_{fp}'(p_0^*)\neq 0$. Restricting the domain, if necessary, we may assume, that the inverse is defined on $[h_1,h_2]$ with $h^*\in(h_1,h_2)$. Define
\begin{equation*}
 \xfp = \pi_{(p,q)}\circ \xfp^H\circ h_{fp}^{-1}\quad \text{and}\quad
 x_{b_i} = \pi_{(p,q)}\circ x_b^H\circ h_{b_i},\quad i=1,2.
\end{equation*}
Shrinking again the interval  $[h_1,h_2]$ if necessary, we may assume, that the range of the above three functions is in $U$. Thus $f_h^2$ defined by (\ref{eq:fh_family}) has period $k$-tupling bifurcation at $s(x^*)$. 
\end{proof}

We end this section by a version of Theorem~\ref{thm:iso} for Hamiltonian systems. Proceeding as in Section~\ref{sec:isochronous} we assume, that (\ref{eq:genHamiltonian}) admits a symmetry $S$ and a reversing symmetry $R$. We also assume that the function defined by (\ref{eq:fpparamH}) satisfies
$$
\Fix(\PM^2,P_0\times P_1\times P_2\times \{0\},\mathcal C_R\wedge\mathcal C_S) = \left\{\xfp^H(p_0): p_0\in P_0\right\}.
$$
Then, we define $G_1^H$ and $g_1^H$ by (\ref{eq:GkH}) and (\ref{eq:gkH}), respectively. We state the following result without a proof, as it is similar to that of Theorem~\ref{thm:iso} and Theorem~\ref{thm:Htupling}.
\begin{theorem}\label{thm:Hiso}
Assume \textbf{HC2--HC4} with $k=1$ and let $x^*=(p_0^*,p_1^*,p_2^H(p_0^*,p_1^*),0)$ be the unique intersection point from Lemma~\ref{lem:uniqueIntersection}. If $S$ and $R$ commute, $\frac{\partial H}{\partial p_0}(x^*)\neq 0$, $h_{fp}'(p_0^*)> 0$ and $h_b''(p_1^*)> 0$, then $f^2_h$ defined by (\ref{eq:fh_family}) has \isoname bifurcation at $s(x^*)$.
\end{theorem}

\section{Bifurcations of odd periodic solutions in the Falkner-Skan equation}
The Falkner-Skan equation \cite{FS} is a third order ODE given by 
\begin{equation}\label{eq:FalknerSkan}
f'''+ff''+\beta\left[1-(f')^2\right]=0. 
\end{equation}
It is low-dimensional and without singularities. Therefore it is relatively easy for rigorous numerical investigation. Although for some physical reasons solutions of certain BVP for (\ref{eq:FalknerSkan}) are relevant, we use the system to test the methodology introduced in Section~\ref{sec:theory} and thus we will focus on periodic solutions. Using this example we will also provide the reader with some details regarding validation of the presence of period-tupling and \tagname bifurcations. A higher-dimensional hamiltonian system (Circular Restricted Three Body Problem), which is also more computationally demanding, will be studied in Section~\ref{sec:cr3bp}. 

Our aim is to prove that some family of odd periodic solutions of (\ref{eq:FalknerSkan}) parametrized by $\beta$ undergoes period-doubling, third order \tagname and period-quadrupling bifurcations. 

The equation (\ref{eq:FalknerSkan}) can be rewritten as a system of first order equations
\begin{equation}\label{eq:FS3d}
 \left\{
 \begin{array}{rcl}
  x'&=&y,\\
  y'&=&z,\\
  z'&=&\beta(y^2-1)-xz 
  \end{array}
 \right.
\end{equation}
and in the sequel we will work with (\ref{eq:FS3d}). The system (\ref{eq:FS3d}) is reversible with respect to $R(x,y,z)=(-x,y,-z)$. For all $\beta>0$ the points $(0,\pm1,0)$ are $R$-symmetric equilibria. Numerical simulation shows, that there is a family of $R$-symmetric periodic orbits $u_\beta(t)=(x_\beta(t),y_\beta(t),z_\beta(t))$ parametrized by $\beta>1$. These periodic orbits intersect the symmetry line $\Fix(R)=\{(0,y,0):y\in\mathbb R\}$ at exactly two points, which approach $(0,\pm1,0)$, respectively, when $\beta\to1^+$ --- see Fig.~\ref{fig:fs-per}. Thus, the period of these orbits goes to infinity when $\beta\to1^+$. It is also observed, that $\max_{t\in\mathbb R}\|x_\beta(t)\|$ goes to infinity when $\beta\to1^+$.
\begin{figure}[htbp]
 \centerline{\includegraphics[width=\textwidth]{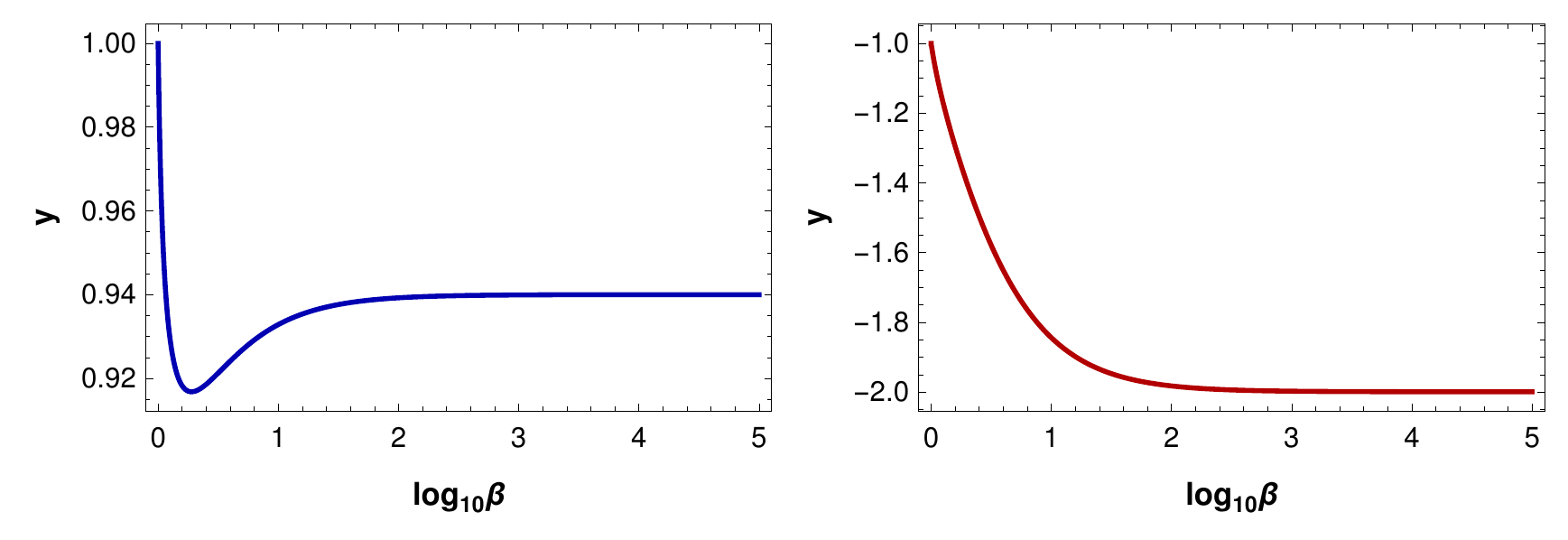}}
 \caption{The $y$ coordinate of two intersection points of $\Fix(R)$ with an approximate $R$-symmetric periodic solution of (\ref{eq:FS3d}) as a function of $\beta\in(1,100\,000]$.\label{fig:fs-per}}
\end{figure}

For small $\beta>1$, these periodic orbits are hyperbolic. For larger $\beta$ they become elliptic and crossing strong $1:k$ resonances, $k=2,3,4$ at approximate parameter values $\hat\beta_k$, respectively, where 
\begin{equation}\label{eq:beta}
\begin{array}{rcl}
 \hat\beta_2 &=& 340.18753498914353231,\\
 \hat\beta_3 &=& 453.442586821384637563,\\
 \hat\beta_4 &=& 679.95415507296894192.
 \end{array}
\end{equation}
Approximate initial conditions for resonant $R$-symmetric periodic orbits are $(0,\hat y_k,0)$, where
\begin{equation}\label{eq:fprim}
\begin{array}{rcl}
 \hat y_2 &=& 0.939792756949623004649,\\
 \hat y_3 &=& 0.939848585715971576498,\\
 \hat y_4 &=& 0.939904499164097608161.
 \end{array}
\end{equation}
and their shapes are shown in Fig.~\ref{fig:resOrbits}. It seems, the family $u_\beta$ continues to exists for all $\beta>1$ approaching $1:1$ resonance when $\beta\to\infty$.
\begin{figure}[htbp]
 \centerline{\includegraphics[width=\textwidth]{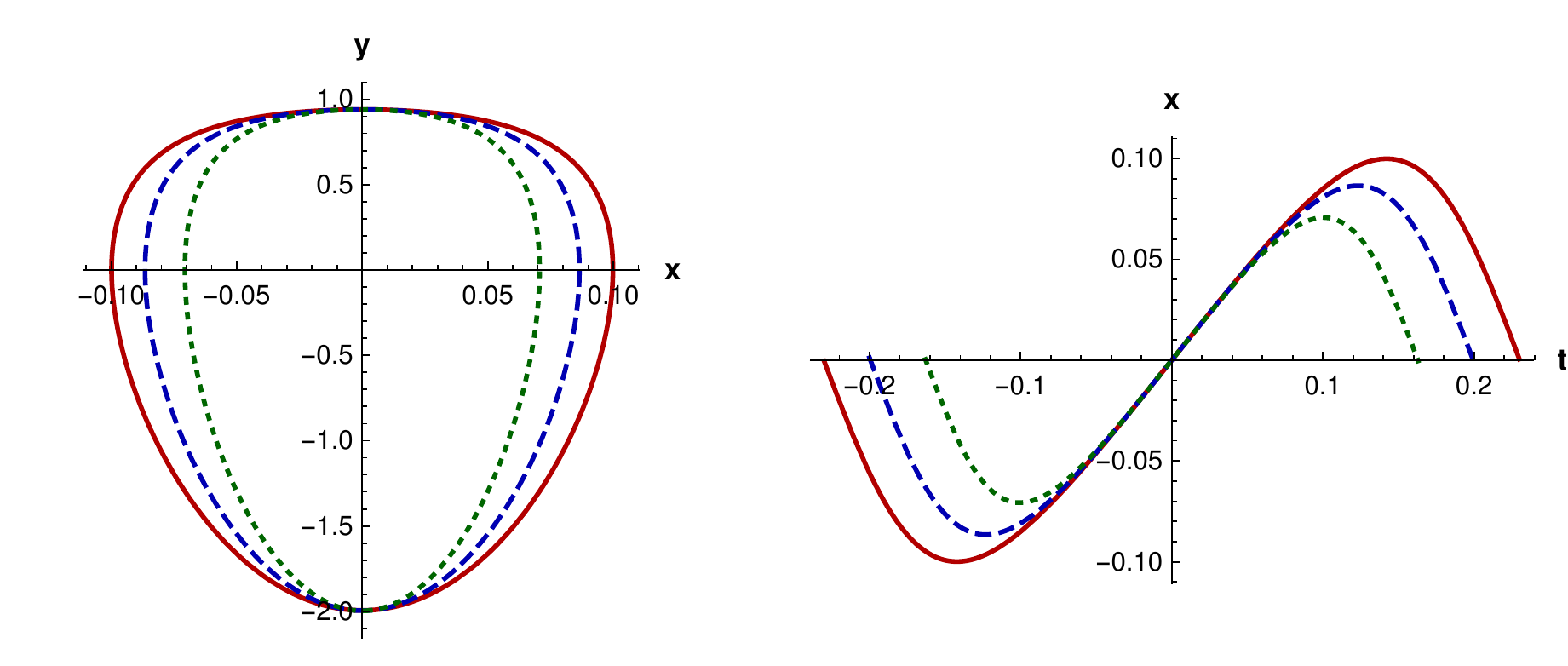}}
 \caption{$R$-symmetric periodic orbits corresponding to $1:2$ (solid), $1:3$ (dashed) and $1:4$ (dotted) resonances.\label{fig:resOrbits}}
\end{figure}

Let us define a Poincar\'e section
$$
\Pi=\{(x,y,z):z=0\}.
$$
We will use $(x,y)$ coordinates to describe points in $\Pi$. By $\PM_\beta:\Pi\to\Pi$ we denote the associated Poincar\'e map for the system (\ref{eq:FS3d}) with the parameter value $\beta$. The map $\PM_\beta$ is reversible with respect to involution $R(x,y)=(-x,y)$. We will also use the notation $\PM(\beta,x,y) = \PM_\beta(x,y)$. 

The aim of this section is to give a computer-assisted proof of the following theorem.
\begin{theorem}\label{thm:FalknerSkan}
 There is a smooth family $u_\beta=(x_\beta,y_\beta)$ of $R$-symmetric period-two points for $\PM_\beta$ parametrized by $\beta\in\mathcal J= [\frac{9}{8},100\,000]$. This family undergoes period-doubling, third order \tagname and period-quadrupling bifurcations at some points $(\beta_k^*,0,y_k^*)$, $k=2,3,4$, respectively, with
 \begin{equation*}
  \beta_k^*\in \mathcal J_k:=\hat \beta_k+|-1,1]\cdot 10^{-9},\qquad y_k^*\in Y_k:=\hat y_k+[-1,1]\cdot 10^{-13},
 \end{equation*}
 where approximate bifurcations points are listed in (\ref{eq:beta})--(\ref{eq:fprim}).
\end{theorem}
\begin{proof}
 The existence of a smooth branch of $R$-symmetric period-two points for $\PM_\beta$ has been proved by means of the method described in \ref{sec:continuation}. We used an adaptive cover of the  parameter range $\mathcal J=\bigcup_{j=1}^{10478}\mathcal J_j$, where the diameters of intervals $\mathcal J_j$ are smaller (approximately $2\cdot 10^{-5}$) if $\mathcal J_j$ is close to $\frac{9}{8}$ and quite large (above $141$) in the second end of the parameter range. Then, for each subinterval $\mathcal J_j$ we check the assumptions of the parametrized Interval Newton Method (Lemma~\ref{lem:ContinuationLemma}). If succeed, we prove that the segments $y(\mathcal J_j)$ glue into a smooth curve by checking the the bounds on $y(\mathcal J_j)$ resulted from the Interval Newton operator overlap, when their domains do.
 
 We will give more details regarding validation of bifurcations. The bifurcation function is
 \begin{equation*}
  G_k(\beta,y) = \pi_x\PM_\beta^k(0,y)
 \end{equation*}
 for $k=2,3,4$. Hence, the Lyapunov-Schmidt reduction is not needed. In order to apply the general framework introduced in Section~\ref{sec:theory} we need first to check conditions \textbf{C2--C4}. We have the following bounds on $G_1$. 
 \begin{equation*}
  \begin{array}{c|c|c|c}
   & G_1(\mathcal J_k\times\{\hat y_k\}) & \frac{\partial G_1}{\partial y}(\mathcal J_k\times Y_k) & -G_1(\mathcal J_k\times\{\hat y_k\})\cdot \frac{\partial G_1}{\partial y}(\mathcal J_k\times Y_k)^{-1}\\\hline
   k=2 & [-1.134,1.22]\cdot 10^{-14} & 0.63412293[27,30] & [-1.8,1.8]\cdot 10^{-14}\\\hline
   k=3 & [-9.19,9.06]\cdot 10^{-15} & 0.549359258[1,4] & [-1.7,1.7]\cdot 10^{-14}\\\hline   
   k=4 & [-6.82,6.73]\cdot 10^{-15} & 0.448707213[7,9] & [-1.5,1.6]\cdot 10^{-14}
  \end{array}
 \end{equation*}
 We see that in each case $\hat y_k-G_1(\mathcal J_k\times\{\hat y_k\})\cdot \frac{\partial G_1}{\partial y}(\mathcal J_k\times Y_k)^{-1}\subset Y_k$, which proves that, there is a branch $(\beta,y(\beta))$ of zeroes of $G_1$ parametrized by $\beta\in\mathcal J_k$, for $k=2,3,4$.
 
 In order to check \textbf{C3} we need bounds on the reduced bifurcation function $g_k$ -- see (\ref{eq:gk}). We have 
 \begin{equation*}
  \begin{array}{c|c|c|c}
   & g_k(\{\hat \beta_k\}\times Y_k) & \frac{\partial g_k}{\partial \beta}(\mathcal J_k\times Y_k) & -g_k(\{\hat \beta_k\}\times Y_k)\cdot \frac{\partial g_k}{\partial \beta}(\mathcal J_k\times Y_k)^{-1}\\\hline
   k=2 & [-9,9]\cdot 10^{-17} & 0.0037227[7,9] & -[2.3,2.33]\cdot 10^{-14}\\\hline
   k=3 & [-1.37,1.36]\cdot 10^{-12} & 0.00364[29,31] & [-3.73,3.74]\cdot 10^{-10}\\\hline   
   k=4 & [-2.32,2.17]\cdot 10^{-15} & 0.001[2961,3416] & [-1.67,1.79]\cdot 10^{-12}
  \end{array}
 \end{equation*}
 Note, that in order to obtain tiny bounds on $g_k(\{\hat \beta_k\}\times Y_k)$ we used high precision interval arithmetic \cite{mpfr}. Again, in each case we have $\hat \beta_k-g_k(\{\hat \beta_k\}\times Y_k)\cdot \frac{\partial g_k}{\partial \beta}(\mathcal J_k\times Y_k)^{-1}\subset \mathcal J_k$, which proves that for $k=2,3,4$ the function $g_k$ has branch of zeroes $(\beta(y),y)$ parametrized by $y\in Y_k$ and thus \textbf{C3} is satisfied.

 From the following estimates 
 \begin{equation}\label{eq:FSC4Bounds}
  \begin{array}{c|c|c|c}
   & \frac{\partial^2 G_k}{\partial \beta\partial y}(\mathcal J_k\times Y_k) & \frac{\partial^2 G_k}{\partial y^2}(\mathcal J_k\times Y_k) & y'(\mathcal J_k)\\\hline
   k=2 & 0.0037227[7,9] & [-1,1]\cdot 10^{-5} & 6.56766[1,3]\cdot 10^{-7}\\\hline
   k=3 & 0.00364[79,81] & -[27.1801,27.1803] & 3.699[899,901]\cdot 10^{-7}\\\hline   
   k=4 & 0.001[2961,3416] & [-0.045,0.045] & 1.64687[86,92]\cdot 10^{-7}
  \end{array}
 \end{equation}
 it follows that 
 \begin{equation*}
  0\notin \frac{\partial^2 G_k}{\partial \beta\partial y}(\mathcal J_k\times Y_k) + \frac{\partial^2 G_k}{\partial y^2}(\mathcal J_k\times Y_k)y'(\mathcal J_k)
\end{equation*}
for $k=2,3,4$ and the condition \textbf{C4} is also satisfied. 

There remains to check conditions specific for each type of bifurcation. In (\ref{eq:FSC4Bounds}) we have already computed bound on $\frac{\partial^2 G_2}{\partial y^2}(\mathcal J_2\times Y_2)$, which proves that the assumptions of Theorem~\ref{thm:tag} are satisfied. Thus, the proof of the existence of third order \tagname bifurcation for $\PM^2$ in $\mathcal  J_k\times \{0\}\times Y_k$ is completed.

For period-doubling and period-quadrupling bifurcations we have to check non-degeneracy condition (\ref{eq:nondegeneracyConditionPT}). From (\ref{eq:FSC4Bounds}) we already have, that for $k=2,4$ there holds  $\frac{\partial^2 G_k}{\partial \beta\partial y}(\mathcal J_k\times Y_k)>0$. Thus, it suffices to check that $\frac{\partial^3 G_k}{\partial y^3}(\mathcal J_k\times Y_k)$ is non-zero. We have the following bounds
\begin{eqnarray*}
\frac{\partial^3 G_2}{\partial y^3}(\mathcal J_2\times Y_2) & \subset & [1372, 1374]\quad\text{and}\\
\frac{\partial^3 G_4}{\partial y^3}(\mathcal J_4\times Y_4) & \subset & -[1860,2047].
\end{eqnarray*}
Eventually, we have to check that the principal periods of bifurcating orbits. The cases $k=2,3$ do not require any computation, as both numbers are primes. For the case $k=4$ we computed 
\begin{equation*}
g_2(\mathcal J_4\times Y_4) \subset 0.4481076[0,7].
\end{equation*} 
This completes the proof.
\end{proof}

\section{Halo orbits in the Circular Restricted Three Body Problem}\label{sec:cr3bp}
In this section we apply the general framework described in Section~\ref{sec:hamiltonian} to the \emph{Circular Restricted Three Body Problem}. First, we will give a short overview of the CR3BP and we list some of its relevant properties. Then, we will give a computer-assisted proof, that the well known families of halo orbits undergo various types of bifurcations.

\subsection{Equations of motion}
The CR3BP is a mathematical model, that describes the motion of a small body with negligible mass under the gravitational influence of two point like big bodies, called primaries, which rotate around their common mass centre on a circle.

Denote by $\mu$ the relative mass of the smaller primary. In a rotating coordinate system centred at the common mass centre of two big primaries, the dynamics of the small particle is governed by the following system of second-order differential equations \cite{KLMR,S}
\begin{equation}\label{eq:cr3bp}
\ddot{x} - 2\dot y = \frac{\partial \Omega_\mu(x,y,z)}{\partial x},\quad
\ddot{y} + 2\dot x = \frac{\partial \Omega_\mu(x,y,z)}{\partial y},\quad
\ddot{z} = \frac{\partial \Omega_\mu(x,y,z)}{\partial z},
\end{equation}
where
\begin{equation*}
\Omega_\mu(x,y,z) = \frac{1}{2}(x^2+y^2) + \frac{1-\mu}{\sqrt{(x+\mu)^2+y^2+z^2}} + \frac{\mu}{\sqrt{(x-1+\mu)^2+y^2+z^2}}.
\end{equation*}
The system is Hamiltonian and it admits a first integral, called the Jacobi constant, which is given by
\begin{equation*}
C_\mu(x,y,z,\dot x,\dot y,\dot z) = 2\Omega_\mu(x,y,z) - (\dot x^2+\dot y^2 + \dot z^2).
\end{equation*}
The hyperplane $ \{(x,y,z=0,\dot x,\dot y, \dot z=0)\}$ is invariant under the local flow induced by (\ref{eq:cr3bp}) and the corresponding four-dimensional Hamiltonian system is called the \emph{Planar Circular Restricted Three Body Problem} (PCR3BP). 

\subsection{Symmetries of the CR3BP}
The system possesses two main symmetries
\begin{equation}\label{eq:CR3BPsymmetries}
\begin{array}{rcl}
S&\colon& (x(t),y(t),z(t))\longrightarrow(x(t),y(t),-z(t)),\\
R&\colon& (x(t),y(t),z(t))\longrightarrow(x(-t),-y(-t),z(-t)),
\end{array}
\end{equation}
It is important to note that $R$ and $S$ commute. This property is required for our method of validation of the existence of \isoname bifurcations --- see Theorem~\ref{thm:main_isochronous}.

%
\subsection{Poincar\'e map in the CR3BP}
Let us define the following Poincar\'e section
\begin{eqnarray*}
\Pi &=& \{(x,y,z,\dot x,\dot y,\dot z)\in\mathbb R^6:y=0\}
\end{eqnarray*}
and denote by $\PM_\mu:\Pi \to \Pi$ the associated Poincar\'e map. We will skip the dependency on $\mu$ and write $\PM$, if it will be clear from the context.
Since the $y$ variable is fixed and equal to zero on the section, we will use $(x,z,\dot x,\dot y,\dot z)$ coordinates to describe points in $\Pi$. With some abuse of notation on $R$,  we will denote by the same letter the reversing symmetry of the system restricted to points on $\Pi$, i.e.
\begin{equation*}
R(x,z,\dot x,\dot y,\dot z) = (x,z,-\dot x, \dot y, -\dot z).
\end{equation*}
Since 
\begin{equation*}
\Fix(R) = \{(x,z, \dot x=0, \dot y, \dot z=0) \, : \, x,z,\dot y \in \mathbb{R} \} \subset \Pi 
\end{equation*}
by \cite[Lemma 3.3]{W} the mapping $\PM$ is reversible, too. Thus, the frameworks for computer-assisted verification of various types of bifurcations introduced in Section~\ref{sec:hamiltonian} can be applied to $\PM$.

\subsection{Periodic orbits near libration points}
The CR3BP possesses five equilibrium points, called the \emph{libration points}. All of them are located in the $\{z=0,\dot z=0\}$ invariant hyperplane and thus they are equilibrium points for the PCR3BP, as well. Three of libration points, commonly denoted by $L_1$, $L_2$ and $L_3$, are collinear and are located on the $x$-axis. They are of saddle-centre type for the PCR3BP. It is well known \cite{S}, that for all $\mu\in(0,1)$ there exists a family of $R$--symmetric periodic orbits, called \emph{planar Lyapunov orbits} (PLO), that surround these libration points --- see also Figure~\ref{fig:bifurcation3d}. In \cite{C,CR} the existence of Lyapunov orbits around $L_1$ and $L_2$ libration points for selected mass parameters has been proved in an explicit domain. Computer-assisted methods have been used to prove \cite{Ar,SK,WZ1,WZ2}, that for the mass parameter $\mu=0.0009537$ corresponding to the Sun-Jupiter system there are Lyapunov orbits around $L_1$ and $L_2$ for certain energy level, and there is countable infinity of connecting orbits between them in both directions. 

For the full system (CR3BP) the libration points $L_{1,2,3}$ are of saddle-centre-centre type, for all $\mu$. In additional direction $z$ there exists a second family of \emph{vertical Lyapunov orbits} (VLO), which are double symmetric both with respect to symmetry $S$ and reversing symmetry $R$ defined by (\ref{eq:CR3BPsymmetries}). These orbits intersect twice the $x$-axis. Therefore, an object (a spacecraft) located nearby one of those orbits will be periodically collinear with the two main primaries. Thus eclipses or shadows are unavoidable for trajectories approaching planar or vertical Lyapunov orbits.  

There is a numerical evidence \cite{DRPKDGV,GM,H}, that a branch of out-of-plane $R$-symmetric orbits, called \emph{halo orbits}, bifurcates from the Lyapunov family. In opposite to planar and vertical Lyapunov orbits, they never cross the $x$-axis, except at the bifurcation point. For large vertical amplitudes $z$ the halo orbits become more and more aligned to the $(y,z)$ plane allowing continuous observation of both primaries without eclipses. Parts of $L_1$-Lyapunov and $L_1$-halo families are shown in Figure~\ref{fig:bifurcation3d} for the relative mass corresponding to the Sun-Jupiter system.

\begin{figure}[htbp]
\centerline{\includegraphics[width=.5\textwidth]{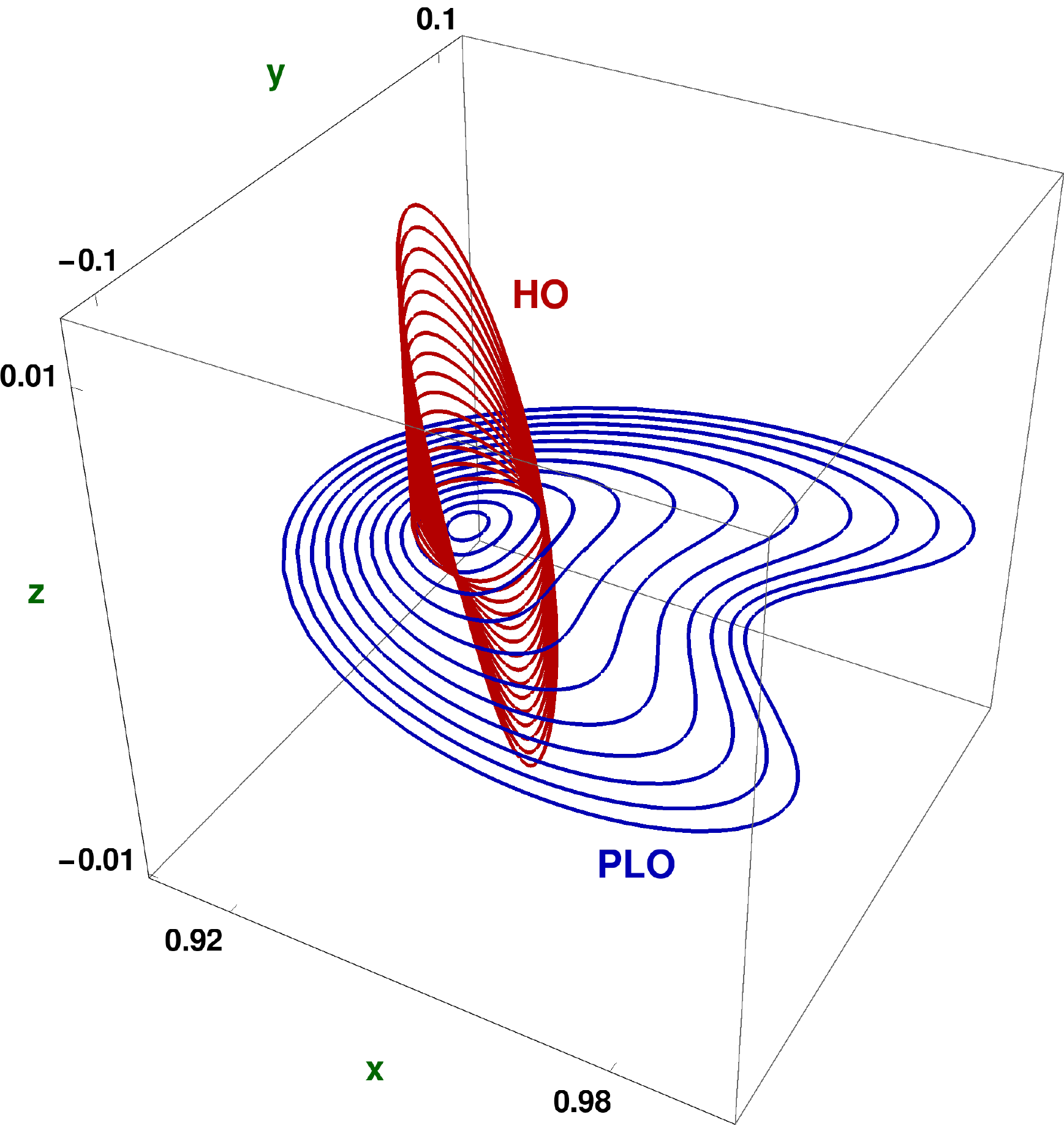}\includegraphics[width=.5\textwidth]{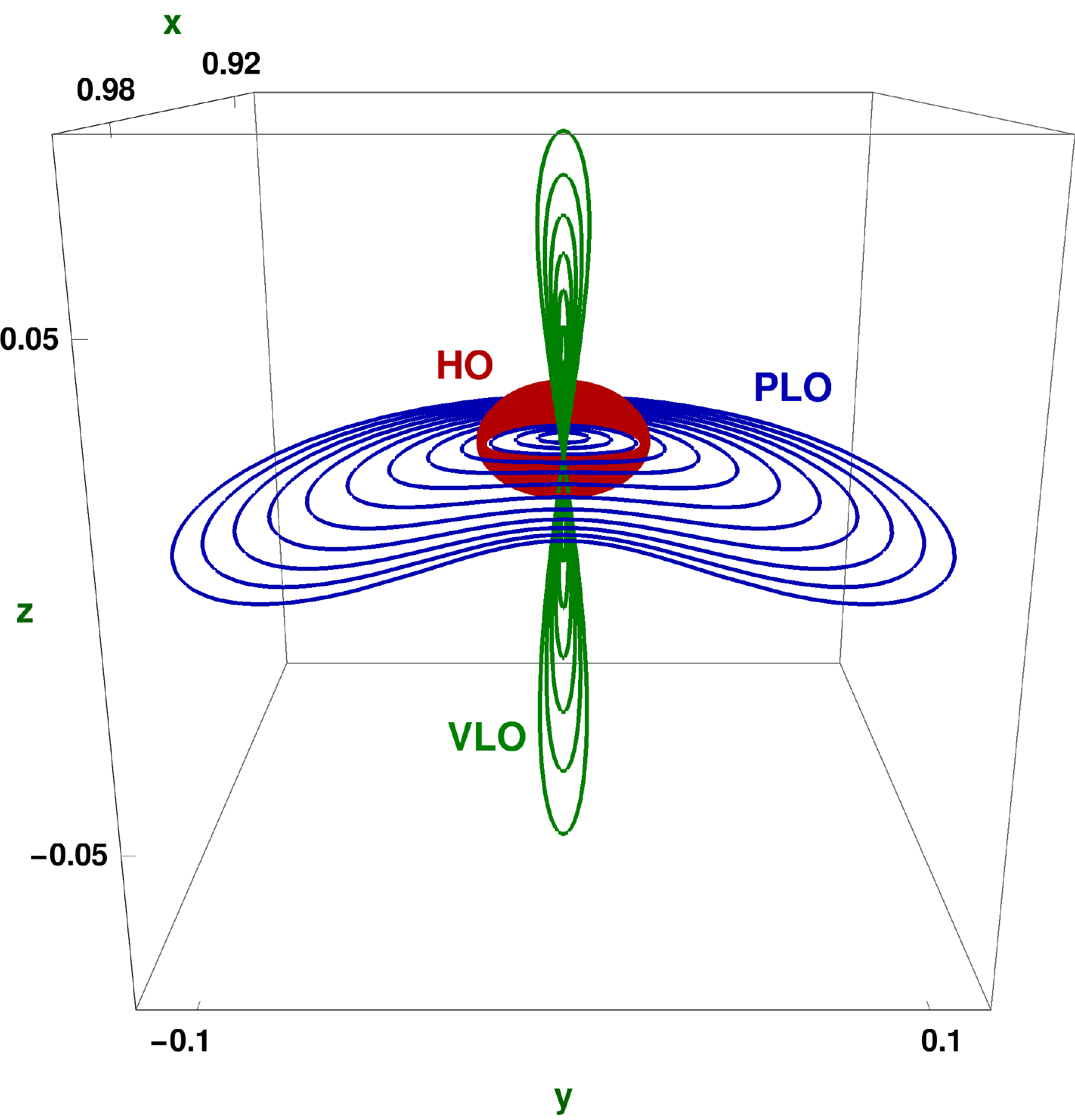}}
\caption{A branch of planar Lyapunov orbits near $L_1$ libration point and bifurcating halo orbits for the mass parameter $\muSJ=9.5388114032796904\cdot 10^{-4}$ corresponding to the Sun-Jupiter system --- see also (\ref{eq:EM_SJ_mass}).\label{fig:bifurcation3d}}
\end{figure}

\subsection{\Isoname bifurcations of halo orbits in the CR3BP}
The halo orbits in the CR3BP are $R$-symmetric, out-of-plane periodic orbits, which bifurcate from the planar Lyapunov family --- see Figure~\ref{fig:bifurcation3d}. Although, there were extensive numerical study of these orbits and their bifurcations (just to mention few papers \cite{H,GM,DRPKDGV}), to the best of our knowledge, they were never proved to exist.  

The best theoretical result in this direction has been done in \cite{CCP,CPS}. The authors consider a~certain normal form, which approximates the CR3BP. They prove, that in the normal form system there is a branch of $R$-symmetric halo orbits bifurcating from $R$-symmetric planar Lyapunov orbits. Moreover, they give an explicit expression for the bifurcation point. This result is valid for all $L_{1,2,3}$ points and for all mass parameters.

The results of this section are complementary to those from \cite{CCP}. We give a computer-assisted proof of the existence of halo orbits for the original CR3BP system for all libration points $L_{1,2,3}$ and for some (not full) range of $\mu$. We will also study continuation and bifurcations of $L_{1,2}$-halo families. 
\begin{theorem}\label{thm:main_isochronous}
  There are continuous functions 
  \begin{equation*}
  h_i\colon [\mu_*,\mu^*]\times Z\ni (\mu,z)\to (x_i(\mu,z),0,z,0,\dot y_i(\mu,z),0)\in \Fix(R)
  \end{equation*}
  with $[\mu_*,\mu^*]=[9.5\cdot10^{-4},0.5]$ and $Z=[-1,1]\cdot \Delta_z$, where $\Delta_z=10^{-7}$, such that 
  \begin{enumerate}
   \item $h_i(\mu,0)$ is a point of \isoname bifurcation of out-of-plane family of $R$--symmetric periodic orbits from the planar family of Lyapunov orbits near $L_i$ libration point and
   \item for $\mu\in [\mu_*,\mu^*]$ and $z\in Z$ the point $h_i(\mu,z)$ is an initial condition for an $R$--symmetric periodic out-of-plane (halo) orbit.
  \end{enumerate}
\end{theorem}
\begin{proof}
We applied the framework from Section~\ref{sec:hamiltonian} to the family of Poincar\'e maps $\PM_\mu$. For fixed $i=1,2,3$ we proceed as follows (we skip dependencies on $i$ to simplify notation). The range of parameters $[\mu_*,\mu^*]$ is initially subdivided into $N$ smaller overlapping subintervals $[\mu_*,\mu^*]=\bigcup_{j=1}^{N}\mu_j$, where $\mu_j=[\underline{\mu_j},\overline{\mu_j}]$. For each $j=1,\ldots,N$ we execute in parallel the following (independent) tasks.
 \begin{enumerate}
  \item We find an approximate bifurcation point $\hat u_j = (\hat x,0,0,\hat{\dot y},0)$ for the mass parameter $\hat \mu_j=\frac{1}{2}(\underline{\mu_j}+\overline{\mu_j})$. For this purpose we use the scheme described in~\ref{sec:findingBifurcationPoints} --- see also Remark~\ref{rem:AppIsoPoint}.
  
  \item The planar double-symmetric Lyapunov orbits can be easily isolated by restriction to the planar system. Using Lyapunov-Schmidt reduction and the method from~\ref{sec:continuation} we validate the existence of smooth, two-parameter families of periodic orbits
  \begin{eqnarray*}
  u_{fp}&\colon&\mu_j\times X\to X\times \{0\}\times\{0\}\times \dot Y\times \{0\},\\
  h_i&\colon&\mu_j\times Z\to  X\times Z\times\{0\}\times \dot Y\times \{0\},
  \end{eqnarray*}
  which correspond to Lyapunov orbits and halo orbits, respectively. The diameters of $X$ and $\dot Y$ were hand-optimized to speed-up computation.
  \item The set $W=X\times \{0\}\times\{0\}\times \dot Y\times \{0\}$ is a bound for the bifurcation point for each $\mu\in\mu_j$. Then we check the assumptions of Theorem~\ref{thm:Hiso}, i.e. $\frac{d}{dx} C_{\mu_j}(W)\neq 0$, $\frac{d}{dx}C_{\mu_j}(u_{fp}(\mu_j,X))\neq 0$ and $\frac{d^2}{dz^2}C_{\mu_j}h_i(\mu_j,z=0)\neq 0$. Note, that in Theorem~\ref{thm:Hiso} we required, that the branching-off curve is convex, but switching of either sign does not change the geometry of overall picture.
 \end{enumerate}
 If any of the above steps fails, the interval $\mu_j$ is being subdivided and we repeat computation for  each element of subdivision. Finally, if all tasks are completed, using the methods from \cite{GLM,BLM} we check, that the pieces of $h_i$ glue into a smooth function. 
 \end{proof}
\begin{figure}[htbp]
\centering
\includegraphics[width=\textwidth]{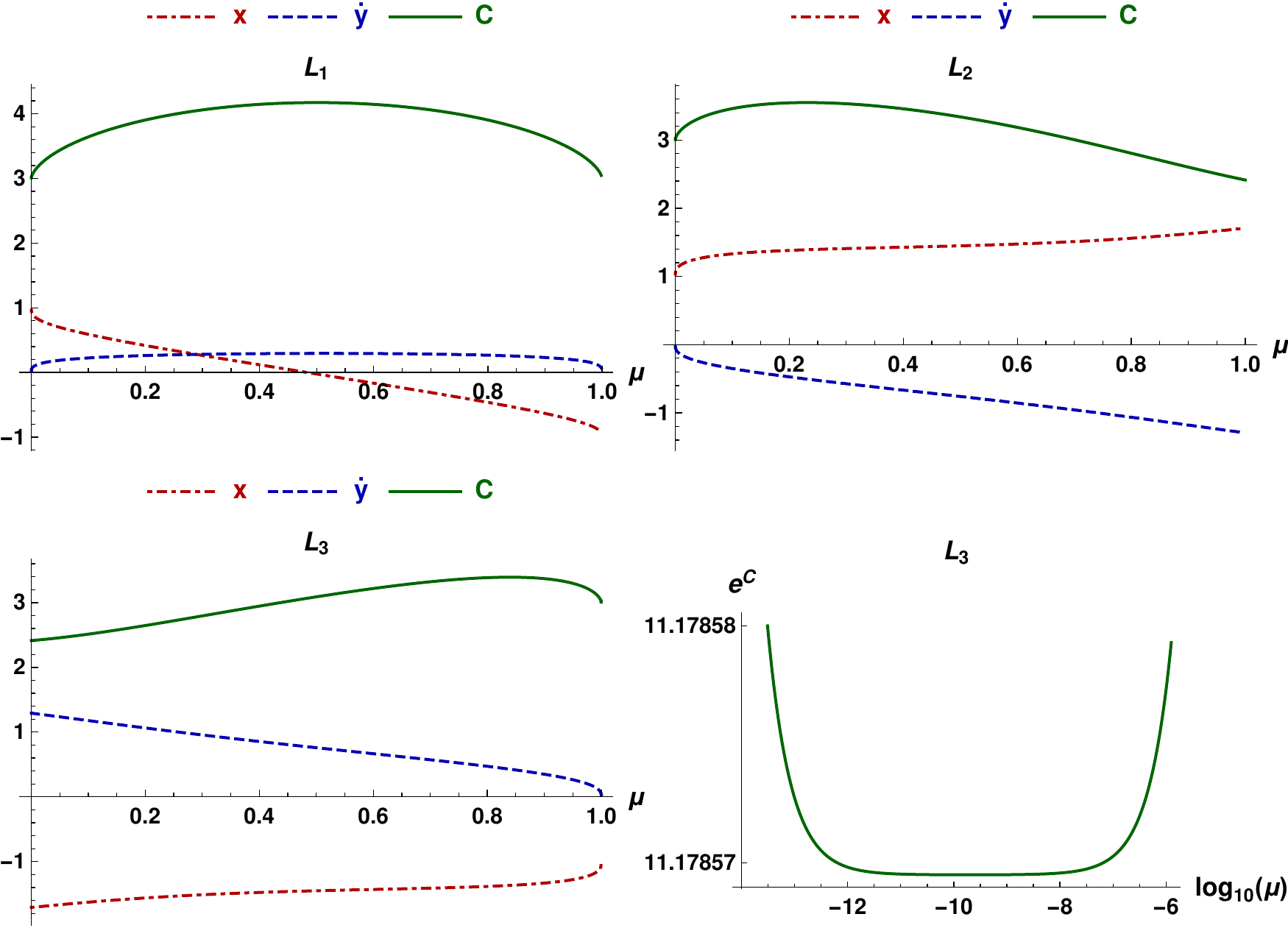}
\caption{Curves of \isoname bifurcation points $(x_i(\mu),\dot y_i(\mu))$ for $L_{1,2,3}$-Lyapunov families and the Jacobi constant $C_\mu(x_i(\mu),0,0,0,\dot y(\mu),0)$. The right-bottom panel indicates, that the Jacobi constant at the bifurcation point as a function of $\mu$ has a local minimum. In order to make this minimum more evident, the plot is shown in $\log$-$\exp$ scale.}
\label{fig:isochoronus}
\end{figure}

Graphs of $\mu\to h_i(\mu,0)$ are shown in Figure~\ref{fig:isochoronus}. We would like to emphasize, that they match quite well bifurcation points coming from the normal forms found in \cite{CCP}. Our validation algorithm used to prove Theorem~\ref{thm:main_isochronous}, by its construction, cannot continue with $\mu\to 0$. The threshold value $\mu_*=9.5\cdot10^{-4}$ as well as size of out of plane amplitude $\Delta_z=10^{-7}$ are by our choice a compromise between CPU time needed to obtain the result and the range of $\mu$ and $z$ we can cover. In particular, the range of mass parameter $[\mu_*,\mu^*]$ contains two relevant values
\begin{equation}\label{eq:EM_SJ_mass}
\muSJ = 9.5388114032796904\cdot 10^{-4},\quad
\muEM = 1.2150584460350998\cdot 10^{-2}
\end{equation}
corresponding to Sun-Jupiter and Earth-Moon systems, respectively. The values listed in (\ref{eq:EM_SJ_mass}) are taken as the nearest IEEE-754 double precision numbers to the recent mass measurements reported in \cite{PS, ASA}. 

\begin{remark}\label{rem:L3monotonocity}
 The $L_3$ is a special case, because for small $\mu$ the maximal order of normal form constructed in \cite{CCP} is $2$, and it is divergent when $\mu\to 0$ \cite{P}. In this case we observe an interesting phenomenon. If $\mu\to 0$, the Jacobi integral seems to be not monotone along the curve of bifurcation points --- see Figure~\ref{fig:isochoronus} right-bottom panel. The computation, which indicate the presence of a local minimum is non-rigorous but performed in high accuracy (400 bits of mantissa) floating point arithmetic \cite{mpfr} and using $80^\mathrm{th}$ order ODE solver with the tolerance per time step set to $10^{-60}$. 
\end{remark}

\subsection{Continuation and bifurcations of halo orbits}

Theorem~\ref{thm:main_isochronous} guarantees that for $\mu\in[9.5\cdot 10^{-4},0.5]$ a branch of halo orbits near $L_i$ can be parametrized by out of plane amplitude $z$ to at least $|z|\leq \Delta_z$. Numerical simulations \cite{DRPKDGV,GM} strongly indicate, that these branches continue to exist for much larger amplitudes and that they undergo period-doubling and period-quadrupling and third-order \tagname bifurcations. The next theorem addresses this issue.

\begin{theorem}\label{thm:L1_circle}
  Consider the CR3BP with $\mu\in\{\muSJ,\muEM\}$ as defined in (\ref{eq:EM_SJ_mass}). There exists a smooth function $h_\mu:\mathcal{S}^1\to \mathbb R^6$ such that for $\tau\in[0,2\pi]$ the following holds true.
  \begin{enumerate}
  \item $h_\mu(\tau)=(x(\tau),0,z(\tau),0,\dot y(\tau),0)$  is an initial condition for an $R$--symmetric periodic (halo) orbit. 
  \item $h_\mu(\tau) = S(h(2\pi-\tau))$ -- the family is $S$ symmetric. 
  \end{enumerate}
  This closed loop of halo orbits intersect the invariant subspace $\{z=0,\dot z=0\}$ at exactly two points $h(0)$ and $h(\pi)$, at which an \isoname bifurcation occurs. 
  Moreover, the branch $h_\mu$ undergoes period doubling, period quadrupling and third order \tagname bifurcations as listed in Table~\ref{tab:L1Tuplings}.
\end{theorem}
\begin{proof}
The branch of halo orbits is split into four pieces.
 \begin{enumerate}
  \item Proceeding as in the proof of Theorem~\ref{thm:main_isochronous} we validate the existence of two \isoname bifurcations: one in the region $x>0$ and the second in $x<0$ --- see Figure~\ref{fig:circles}. From \textbf{HC3} we also have, that there is a branch of out-of-plane $R$-symmetric periodic orbits parametrized by $z\in [-\Delta_z,\Delta_z]$, for an explicit $\Delta_z>0$. 
  \item Both out-of-plane families are then continued (see \ref{sec:continuation}) and parametrized by $z$ variable until hand-chosen threshold value $z=0.625$, as shown in Figure~\ref{fig:circles}. 
  \item The upper arc joining two halo orbits with $z=0.625$ is parametrized by $x$ variable.
 \end{enumerate}
By the well known techniques \cite{BLM} we can check, that these pieces glue into a smooth curve. By the symmetry we obtain the lower branch of halo orbits. Summarizing, we obtained, that the branch of halo orbits is a compact, smooth, one-dimensional manifold without boundary. It is well known \cite{Mi}, that such a manifold is diffeomorphic to a circle. By the construction the manifold is $S$-symmetric, thus the parametrization $h$ can be chosen to preserve this symmetry, as well. 

Approximate bifurcation points listed in Table\ref{tab:L1Tuplings} (except for $k=1$ were found with very high accuracy (of order $10^{-60}$) using high-order ODE solvers from the CAPD library \cite{CAPD} based on high precision floating-point arithmetic \cite{mpfr}. Then we checked assumptions of Theorem~\ref{thm:Htag} and Theorem~\ref{thm:Htupling} on very small sets (of the size about $10^{-30}$) centred at these approximate bifurcation points. 
Notice, that in one case $j=7$ and $\mu=\muSJ$ we could not obtain bounds on third order derivatives sharp enough to check convexity of Jacobi integral along bifurcation curve. We checked, however, the conditions \textbf{HC2--HC4}, which in particular means, that there are two curves of halo orbits of period $1$ and $4$ intersecting at a single point. 
\end{proof}
\begin{remark}
 Although we did not prove it, there is a strong numerical evidence that the points $h_\mu(0)$ belongs to the family of $L_1$-Lyapunov orbits. A possible method to close this gap is to apply the method for computation of invariant manifolds of Lyapunov orbits, as proposed in \cite{CR}.
\end{remark}

Projections of the two curves $h_\mu(\mathcal S^1)$ for $\mu\in\{\muSJ,\muEM\}$ resulting from Theorem~\ref{thm:L1_circle} onto $(x,z)$ plane are shown in Figure~\ref{fig:circles}. These families form a $2D$--tori in the full phase space. Such torus for the mass $\mu=\muSJ$ is shown in Figure~\ref{fig:L1JupiterTori}.
\begin{figure}[htbp]s
\centering
\includegraphics[width=\textwidth]{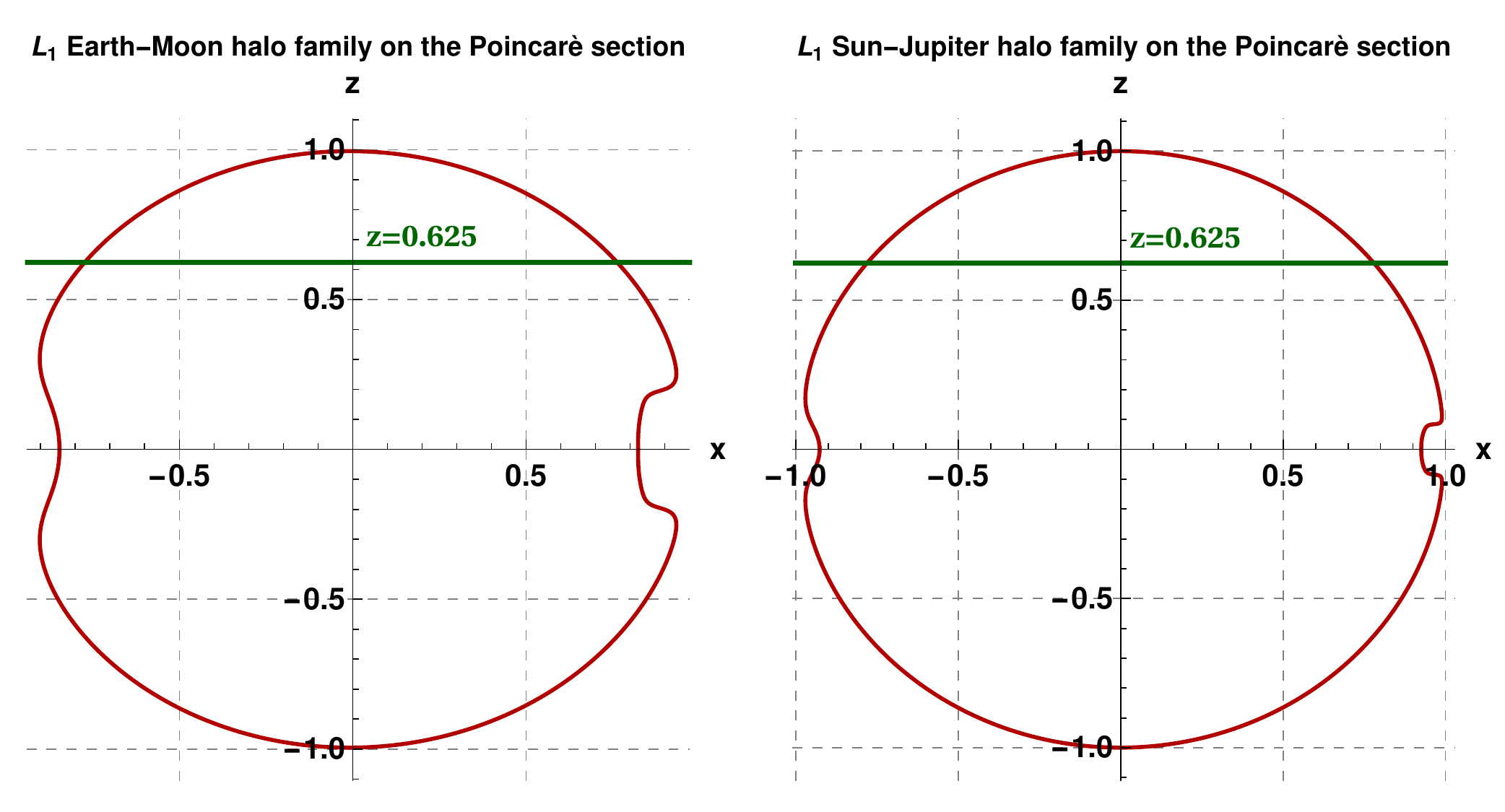}
\caption{Projection onto $(x,z)$-plane of $h_\mu(\mathcal S^1)$, $\mu\in\{\muSJ,\muEM\}$ as defined in Theorem~\ref{thm:L1_circle}. Each point $h_\mu(\tau)$ is an initial condition for a halo orbit, thus the entire family forms a $2D$-tori in the full phase space -- see Figure~\ref{fig:L1JupiterTori}.}
\label{fig:circles}
\end{figure}

\begin{figure}[htbp]
\centering
\includegraphics[width=.8\textwidth]{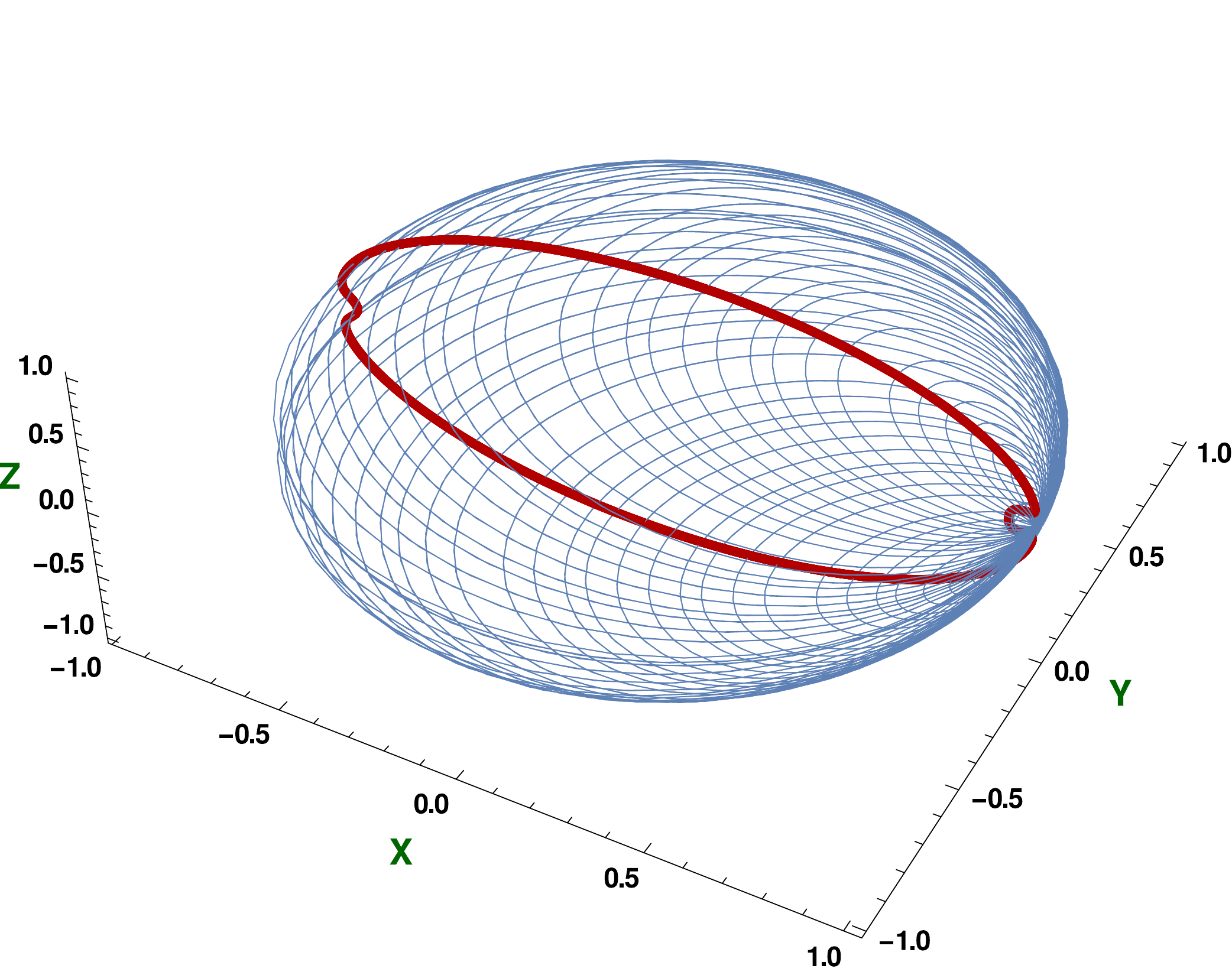}
\caption{$2D$ torus of primary $L_1$-halo orbits in the Sun-Jupiter system.}
\label{fig:L1JupiterTori}
\end{figure}

\begin{figure}[htbp]
\centering
\includegraphics[width=\textwidth]{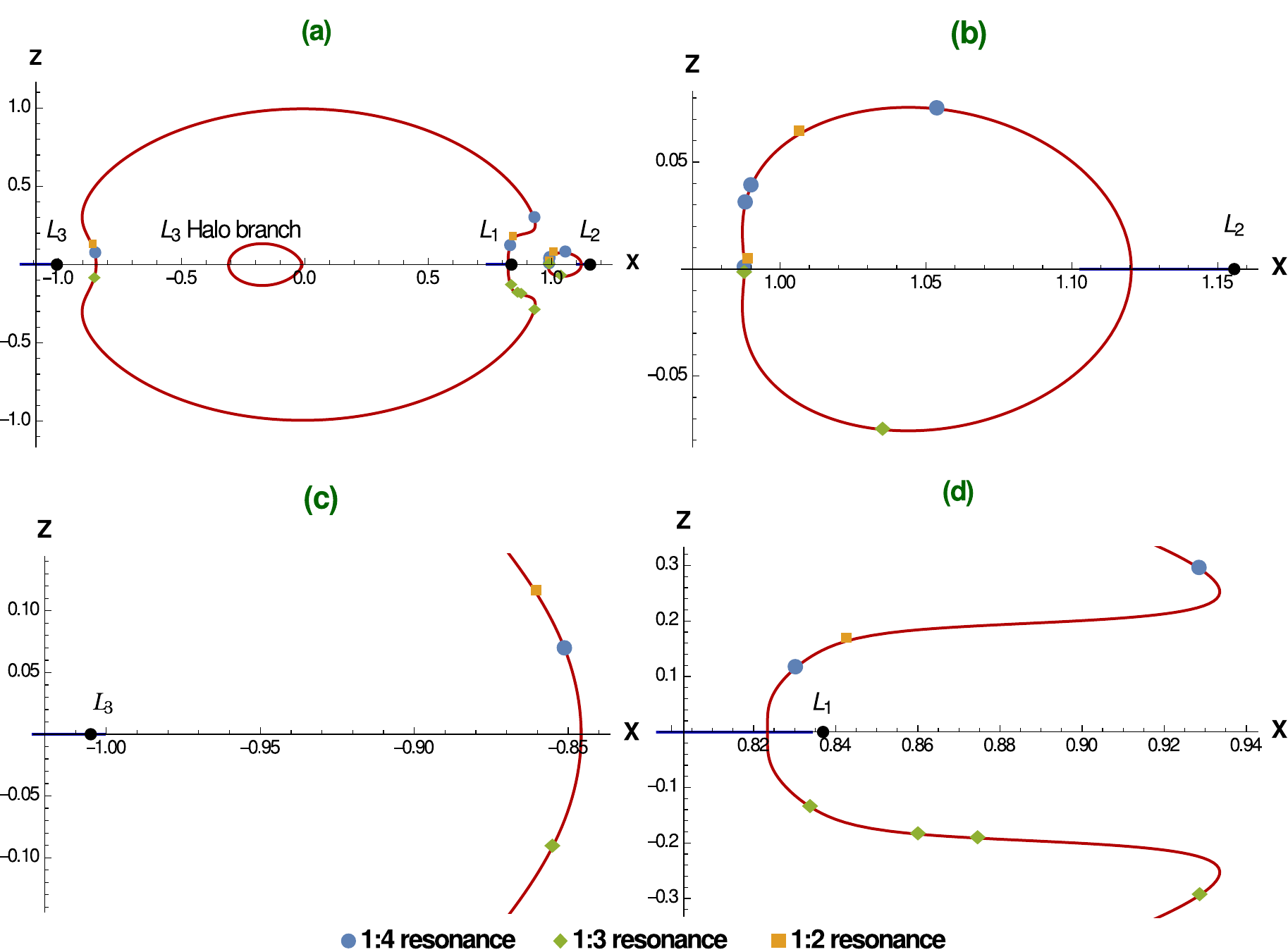}
\caption{Intersection of families of $L_{1,2,3}$-halo orbits in the Earth-Moon system with the Poincar\'e section $\Pi$ along with strong resonances, that lead to period-doubling and period-quadrupling bifurcations (marked by squares and disks in $z>0$ half-plane, respectively) and \tagname bifurcations (marked by diamonds in $z<0$ half-plane), as proved in  Theorem~\ref{thm:L1_circle} and Theorem~\ref{thm:L2_circle}. 
(a) Mutual location of $L_{1,2,3}$-halo branches: $L_{1,2}$-branches are shown on $\{y=0,\dot y>0\}$ section, $L_3$ is shown on $\{y=0,\dot y<0\}$ section; (b) location of strong resonances on $L_2$-halo branch; (c) and (d) location of strong resonances on $L_1$-halo branch.
}
\label{fig:L23Circles}
\end{figure}

Numerical simulation \cite{H,DRPKDGV} shows that $L_{2,3}$-halo families continue to exists until a collision with one of the primaries --- see Figure~\ref{fig:L23Circles}. Hence, on the Poincar\'e section $\Pi$, $(x,z)$ coordinates of these orbits approach $(1-\mu,0)$ and $(\mu,0)$ respectively, while $|\dot y|$ tends to infinity.

We have the following partial result for the $L_2$-halo family.
\begin{theorem}\label{thm:L2_circle}
  Consider the CR3BP with $\mu\in\{\muEM,\muSJ\}$ as defined in (\ref{eq:EM_SJ_mass}). There is a smooth function $h_\mu:[-1,1]\to \mathbb R^6$ such that for $\tau\in[-1,1]$ the following statements hold true.
  \begin{enumerate}
  \item $h_\mu(\tau)=(x_\mu(\tau),0,z_\mu(\tau),0,\dot y_\mu(\tau),0)$  is an initial condition for an $R$--symmetric periodic (halo) orbit. 
  \item $h_\mu(\tau) = S(h(-\tau))$ -- the family is $S$ symmetric.
  \item $h_\mu(0)$ is a point of an \isoname bifurcation.
  \item The function $\dot y_\mu(\tau)$ has a unique local minimum at $\tau=0$ satisfying
  \begin{equation}\label{eq:BoundOnDYAtMinimum}
  \dot y_{\muEM}(0)\in 0.176040[3,5],\qquad
  \dot y_{\muSJ}(0)\in 0.069870[3,8].
  \end{equation}  
  \item The branches continue to at least 
  \begin{equation}\label{eq:dytreshhold}
  \dot y_{\muEM}(\pm1)= 20.5,\qquad  \dot y_{\muSJ}(\pm1)=8.
  \end{equation}
  \end{enumerate}
  Moreover, these branches undergo period doubling, period quadrupling and third order \tagname bifurcations as listed in Table~\ref{tab:L2Tuplings}.
\end{theorem}
\begin{proof}
 The validation is split into three steps.
 \begin{enumerate}
  \item From Theorem~\ref{thm:main_isochronous} we know, that there is an \isoname bifurcation of halo orbits from $L_2$-Lyapunov family. The estimates (\ref{eq:BoundOnDYAtMinimum}) are taken from this proof. The branching of family of halo orbits $h_\mu(z) = (x_\mu(z),0,z,0,\dot y_\mu(z),0)$ is parametrized by $z\in[-1,1]\cdot \Delta_z$, with $\Delta_z=10^{-7}$. We also check that $\dot y_\mu''(z)>0$ for $|z|\leq \Delta_z$.
  \item The branch is then rigorously continued using \ref{sec:continuation} and parametrized by $h_\mu(z) = (x_\mu(z),0,z,0,\dot y_\mu(z),0)$, $z>0$ until some hand-chosen threshold value of $\hat z$ (dependent on $\mu$). We also checked that for $z\in[\Delta z,\hat z]$ there holds $\dot y_\mu'(z)>0$.
  \item Further continuation of the branch starting from $h_\mu(\hat z)$ is parametrized by $\dot y$ until hand-chosen threshold values (\ref{eq:dytreshhold}). 
 \end{enumerate}
Summarizing, in each segment the variable $\dot y$ is increasing along the branch of periodic orbits which makes it possible to re-parametrize the curve as a function 
$$h_\mu(\tau)=(x_\mu(\tau),0,z_\mu(\tau),0,\dot y_\mu(\tau),0)$$
defined on $\tau\in[0,1]$. From the symmetry $S$ we obtain the second branch for $\tau\in[-1,0]$.

The bifurcations listed in Table~\ref{tab:L2Tuplings} (except $j=0$) were validated using Theorem~\ref{thm:Htupling} and Theorem~\ref{thm:Htag} in high-precision interval arithmetics. Notice, that in two cases $j=8$ and $\mu=\{\muSJ,\muEM\}$ we could not obtain bounds on third order derivatives sharp enough to check convexity of Jacobi integral along bifurcation curve. We checked, however, the conditions \textbf{HC2--HC4}, which in particular means, that there are two curves of halo orbits of period $1$ and $4$ intersecting at a single point.  
\end{proof}

\begin{remark}
The threshold values in (\ref{eq:dytreshhold}) have been chosen so that the validated arc of $L_{2}$--halo orbits contains all strong resonances at which period-tupling and \tagname bifurcations occurs for both values of $\mu\in\{\muSJ,\muEM\}$. Numerical simulation shows, that for larger values of $\dot y$ strong resonances are not present. However, we did not validate this conjecture. A possible approach to close this gap and obtain a full picture of what happens to $L_{2,3}$--halo families is to perform Levi-Civita regularisation \cite{S} and continue branches of halo orbits in these coordinates.
\end{remark}

\begin{table}
\caption{Bound on Jacobi constant for period $k$-tupling and \tagname bifurcations of $L_1$ halo families for the mass parameters $\muEM$ and $\muSJ$. Multiplicity $k=1$ corresponds to \isoname bifurcations. Multiplicity $k=3$ stands for \tagname bifurcations. In the case $j=7$ and $\mu=\muSJ$, denoted by a star, the convexity condition in Theorem \ref{thm:Htupling} has not been checked. We could not obtain bounds on third order derivatives so sharp, which would guarantee that the second derivative of Jacobi integral along bifurcation curve is non-zero.
}
\label{tab:L1Tuplings}
\centering
\begin{tabular}{llll}
\hline\noalign{\smallskip}
j & k & bound on Jacobi constant ($\muEM$) & bound on Jacobi constant ($\muSJ$)\\
\noalign{\smallskip}\hline\noalign{\smallskip}
0 & 1 & 3.17435[03, 36] & 3.03588[11,51] \\
1 & 4 & 3.08384097317512242430038839[79,91] & 3.01979992774088569676042962[45,59]\\
2 & 3 & 3.058886412529835176423178819[4,6] & 3.015412945713342018918305256[4,5]\\
3 & 2 & 3.0216192479264201986830801047[6,8] & 3.00909434962110748263791018519[3,9]\\
4 & 3 & 2.999986911642456326104353768[3,8] & 3.00589961847845578773000478[64,75]\\
5 & 3 & 2.997919501600216512922520[69,71] & 3.00584514988954760426886596[07,61]\\
6 & 3 & 2.94132864491556775199[28,32] & 2.9941342902214648929134[15,26]\\
7 & 4 & 2.940683922766931384[68,79] & 2.9940756819941148370203478568[3,4]\textbf{*}\\
8 & 2 & -0.986509091038502895183600231[6,9] & -0.99874596801401641137454972[14,21]\\
9 & 3 & -0.996795335128162658942078721[1,5] & -1.0006666037342203004067305[44,52]\\
10& 4 & -1.004727349648878143369879[07,22] & -1.00227845861825336488127[00,44]\\
11& 1 & -1.016[09,14] & -1.00460[55,77] \\
\noalign{\smallskip}\hline
\end{tabular}
\end{table}

\begin{table}
\caption{Bound on Jacobi constant for period $k$-tupling and \tagname bifurcations of $L_2$ halo families for the mass parameters $\muEM$ and $\muSJ$. Multiplicity $k=1$ corresponds to \isoname bifurcation. Multiplicity $k=3$ stands for \tagname bifurcations. In two cases $j=8$ and $\mu=\{\muSJ,\muEM\}$, denoted by a star, the convexity condition in Theorem \ref{thm:Htupling} has not been checked. We could not obtain bounds on third order derivatives sharp enough, which would guarantee that the second derivative of Jacobi integral along bifurcation curve is non-zero.}
\label{tab:L2Tuplings}
\centering
\begin{tabular}{llll}
\hline\noalign{\smallskip}
j & k & bound on Jacobi constant ($\muEM$) & bound on Jacobi constant ($\muSJ$)\\
\noalign{\smallskip}\hline\noalign{\smallskip}
0 & 1 & 3.15211[87,91] & 3.03413[67,72] \\
1 & 4 & 3.071869946146936057096946[46,52] & 3.01887850935398942392012584[04,38]\\
2 & 3 & 3.05083503865220863946099978[87,93] & 3.014802161770970024871372364[53,86]\\
3 & 2 & 3.0229911591596336379776897124[1,5] & 3.0091557491590008844773187520[3,6]\\
4 & 4 & 3.0158978159595140970471[78,93] & n/a \\
5 & 2 & 3.017143662542155479781194411[85,94] & n/a\\
6 & 4 & 3.0190035761795315910790[19,29] & n/a \\
7 & 3 & 3.077836508502735302[69,73] & 3.019542146861187965925873562[20,98]\\
8 & 4 & 3.104289978034786552471088092[8,9]\textbf{*} & 3.024350541079342605537853984[5,6]\textbf{*}\\
\noalign{\smallskip}\hline
\end{tabular}
\end{table}

\subsection{Implementation notes}
The source code of all programs is available to download from the web page of the corresponding author \cite{Wweb}. 
The programs are written in C++-11 and use rigorous ODE solvers and algorithms for computation of Poincar\'e maps and their derivatives \cite{WW,WZ4} from the CAPD library \cite{CAPD}. 
All programs the from the were compiled using g++-4.9.2 and executed on a computer equipped with Intel Xeon E7-8867 v4 \@ 2.40GHz processors (64 cores).

\appendix
\section{Interval Newton method for implicit equations}\label{sec:continuation}
In this section we provide a method for validated computation of implicit functions. Such an algorithm is needed to check assumptions \textbf{C2--C3} or \textbf{HC2--HC3}, as well as to compute wide branches of periodic orbits far from bifurcation points. The method is an adaptation of the well known \emph{Interval Newton Method} (INO) \cite{M,N} to the case of implicit equations. The main modification which significantly improves the method is the use of higher order derivatives. The method itself is quite straightforward but to the best of our knowledge, it has not appeared in the literature.
First we recall the Interval Newton Method \cite[Thm.~8.4]{M}, \cite[Thm.~5.1.7]{N}.
\begin{theorem}[\cite{M,N}]\label{thm:IntervalNewtonMethod}
Let $f:\mathbb R^n\to\mathbb R^n$ be a $\mathcal C^1$ map and let $X\subset \mathbb R^n$ be a convex, compact set. For $x_0\in \mathrm{int}X$ we define the interval Newton operator by
\begin{equation*}
N(f,x_0,X) = x_0 - [Df(X)]_I^{-1}f(x_0),
\end{equation*}
where by $[Df(X)]_I$ we mean a convex hull of the set of matrices $\{Df(x): x\in X\}$.
\begin{enumerate}
\item If $N(f,x_0,X)\subset \mathrm{int}\,X$ then $f$ has a unique zero in $X$ that belongs to $N(f,x_0,X)$.
\item If $N(f,x_0,X)\cap X=\emptyset$ then $f$ has no zeros in $X$.
\end{enumerate}
\end{theorem}

The above theorem can be used to validate the existence of solutions to implicit equations over an explicit domain.
\begin{lemma}\label{lem:ContinuationLemma}
Let $f:\mathbb R^m\times\mathbb R^n\to\mathbb R^n$ be a $\mathcal C^1$ map, $Z\subset \mathbb R^m$ be the closure of an open set and let $X\subset \mathbb R^n$ be a convex, compact set with non-empty interior. For $x_0\in \mathrm{int}\,X$ we define the interval Newton operator by
\begin{equation}\label{eq:ParameterizedNewton}
N(f,x_0,X,Z) = x_0 - [D_xf(Z,X)]_I^{-1}f(Z,x_0).
\end{equation}

If $N(f,x_0,X,Z)\subset \mathrm{int}\,X$ then there exists a $\mathcal C^1$ smooth function $g:Z\to X$ such that the set of zeroes $\{(z,x)\in Z\times X : f(z,x)=0 \}$ coincides with the graph of the function $g$, i.e.
$$\left\{(z,x)\in Z\times X : f(z,x)=0 \right\}\, =\, \left\{(z,g(z)) : z\in Z\right\}.$$
\end{lemma}
\begin{proof}
Let us fix $z\in Z$ and put $f_z = f(z,\cdot)$. Then we have
$$
N(f_z,x_0,X)\subset N(f,x_0,X,Z)\subset \mathrm{int}\, X.
$$
From Theorem~\ref{thm:IntervalNewtonMethod} for all $z\in Z$ there exists a unique $x=g(z)\in X$ such that $f(z,x)=0$. Let us observe, that the condition (\ref{eq:ParameterizedNewton}) implies, that $D_xf(z,x)$ is invertible at each point $(z,x=g(z))$. By the implicit function theorem the function $g$ is smooth, as it solves an implicit equation $f(z,g(z))\equiv 0$ at every point $z\in Z$.
\end{proof}

The efficiency of the INO (Theorem~\ref{thm:IntervalNewtonMethod}) is hidden in the fact, that we usually have a very good approximation (from numerical experiments) for zero, i.e. $f(x_0)\approx 0$. Then the quantity $[Df(X)]_I^{-1}f(x_0)$ can be very tight, even if the computed bound on derivative $Df(X)$ is overestimated. This is not the case for parametrized maps, as equation (\ref{eq:ParameterizedNewton}) contains the term $f(Z,x_0)$. If $Z$ is large then having $f(z,x)\approx 0$ for all $z\in Z$ is rather unlikely.

A straightforward way to overcome this problem is to make a substitution $(z,x)=s(z,w)$, such that in the new coordinates the function $z\to w(z)$, which solves the implicit equation $f(s(z,w))=0$, is flat. Let us fix $(z_0,x_0)\in Z\times X$, such that $f(z_0,x_0)\approx 0$ and assume $D_xf(z_0,x_0)$ is non-singular. We define an affine substitution by
$$
(z,x) = s(z,w) := (z,x_0+w-A(z-z_0)),
$$
where $A = D_xf(z_0,x_0)^{-1}D_zf(z_0,x_0)$. The map $s$ is invertible because its linear part has determinant equal to one and thus zeroes of $f$ are in one-to-one correspondence with zeroes of $g:=f\circ s$. In the new coordinates the point $(z_0,w_0=0)$ is an approximate zero of $g$. Moreover, $D_zg(z_0,w_0)= 0$ (provided $A$ is computed exactly). The above idea of changing linearly coordinate system has been proposed in \cite{BLM}, but the method for validation of a branch of zeroes was different and based on so-called radii polynomial approach \cite{GLM,BLM}.

In the remaining part of the section we will show, how to efficiently evaluate all the terms, which appear in the INO for the mapping $g$. Our test show, that using this approach we could significantly reduce overestimations in evaluation of the INO, which lead to significant advantage of this approach in comparison to direct evaluation of INO for $g$.

The INO for the mapping $g$ on the set $Z\times W$, which contains $(z_0,w_0)$ reads
\begin{equation*}
N(g,w_0,Z,W) = - [D_wg(Z,W)]_I^{-1}g(Z,w_0).
\end{equation*}
In what follows, we will show how we can bound all the terms that appear in this expression. Let $X$ be such that $s(Z,W)\subset (Z,X)$ and denote $\Delta Z = Z-z_0$.
The term $g(Z,w_0)$ can be bounded by means of the mean value theorem
\begin{eqnarray}
g(Z,w_0) &\subset& g(z_0,w_0) + [D_zg(Z,w_0)]_I\cdot\Delta Z \nonumber\\
  &=& f(z_0,x_0) + [D_zg(Z,w_0)]_I\cdot\Delta Z\label{eq:gZBound}
\end{eqnarray}
and the set of matrices $D_zg(Z,w_0)$ can be bounded by 
\begin{equation}\label{eq:gZBound1}
D_zg(Z,w_0) \subset [D_zf(Z,x_0)] - 
   \left[D_xf(Z,x_0)A\right]_I\cap \left[\left(D_xf(Z,x_0)D_xf(z_0,x_0)^{-1}\right)D_zf(z_0,x_0)\right].
\end{equation}
By the choice of $A$, we have $D_zg(z_0,w_0)=0$. Therefore we expect that for not very large parameter radius $\Delta Z$, the term $D_zg(Z,w_0)\cdot\Delta Z$ in (\ref{eq:gZBound}) is a small box around zero, as desired. Note, that the above considerations hold true for any matrix $A$. Therefore the quantities $D_xf(z_0,x_0)^{-1}$ and $D_zf(z_0,x_0)$ can be computed using just floating point arithmetic, however their product $A$ must be bounded rigorously, if we want to take the intersection in (\ref{eq:gZBound1}).

The quantity $g(Z,w_0)$ can be also bounded using second order Taylor expansion
\begin{equation}\label{eq:gZBound2}
g(Z,w_0) \subset f(z_0,x_0) + D_zg(z_0,w_0)\cdot\Delta Z + \frac{1}{2}\Delta Z^TD^2_{zz}g(Z,w_0)\Delta Z.
\end{equation}
Recall, the point $(z_0,x_0)$ is chosen so that $f(z_0,x_0)\approx 0$ and the substitution $s$ is chosen so that $D_zg(z_0,w_0)\approx 0$  (note, usually it cannot be computed exactly). Since both quantities are evaluated at a single point, we can take advantage (if necessary) of high-precision interval arithmetic and make these terms as close to zero as desired. Therefore, the bound on $g(Z,w_0)$ is practically quadratic in the radius of the parameter range $\Delta Z$. One can take the intersection of direct evaluation of $g(Z,w_0)$ in interval arithmetic with the bounds obtained from (\ref{eq:gZBound}) and (\ref{eq:gZBound2}).

In order to compute the Interval Newton Operator for $g$ we have to bound $D_wg(Z,W)$. Direct evaluation gives 
\begin{equation}\label{eq:boundDwg}
D_wg(Z,W)\subset D_xf(Z,X).
\end{equation}
Using second order derivatives of $f$ we can obtain another enclosure
\begin{equation*}
D_wg(Z,W)\subset D_xf(z_0,x_0) + D^2_{ww} g(Z,W)W + D^2_{zw} g(Z,W)\Delta Z,  
\end{equation*}
which can be intersected with (\ref{eq:boundDwg}). 

Numerical experiments we performed show that using the above approach we can usually validate the existence of solution to the implicit equation on much wider domain $Z$ (without subdivision) than in the original coordinates. 

In principle, both $g(Z,w_0)$ and $D_wg(Z,W)$ can be bounded using higher order Taylor expansions. This should come along with nonlinear (usually polynomial) substitution $s$, such that all derivatives of $g=f\circ s$ with respect to $z$ vanish at $(z_0,w_0)$ up to desired order $r$. Then the bound on $g(Z,w_0)$ can be made of order $O(\|\Delta Z\|^{r+1})$.   

In the context of the CR3BP we have found, that the second order expansion is very efficient. Note, that computation of higher order derivatives of Poincar\'e maps in a high dimensional system is costly --- the complexity of the $\mathcal C^r$--Lohner algorithm \cite{WZ4} used to integrate variational equations is $O(n^3r^ns^2)$, where $n$ is the dimension, $s$ is the order of Taylor method and $r\geq 1$ is the largest order of derivative of Poincar\'e map, we request. Clearly, increasing $r$ in a high-dimensional system is very expensive. Secondly, the bounds on higher order derivatives of Poincar\'e map are usually much overestimated than those of lower order. 

\section{Newton like scheme for finding bifurcation points}\label{sec:findingBifurcationPoints}
A straightforward way to localize period-tupling and \tagname bifurcation points of a~family of reversible maps $f_\nu(x)$ is to follow the branch of period-$2$ points $x(\nu)$ and look for resonant eigenvalues of $Df^2_\nu(x(\nu))$. In low-dimensional systems one can look at the stability parameter \cite{GM,H}. This method is very efficient when we want to find a rough approximation to the bifurcation point. In reversible or hamiltonian case multiple eigenvalues may occur making computation of eigenvalues with high-accuracy quite non-trivial task.

In this short section we propose eigenvalue-independent yet efficient scheme for finding very accurate approximation to bifurcation points. In what follows we focus on reversible Hamiltonian systems and use the notation from Section~\ref{sec:hamiltonian}, but the idea applies to any family of reversible maps.

Following Section~\ref{sec:hamiltonian} we assume that $(p_0,p_1,p_2,q=0)\in \Fix(R)$ is an approximate period-$2$ point for the Poincar\'e map $\PM$, which is close to $1:k$ resonance. We would like to refine it by a Newton-like scheme. Since we have $n+1$ unknowns $(p_0,p_1,p_2)$ we need $n+1$ equations with expected isolated zero. We impose
\begin{equation}\label{eq:bifPointsScheme}
\widetilde{\mathcal P}(p_0,p_1,p_2) := \pi_q \mathcal P(p_0,p_1,p_2,0)=0\qquad \text{and}\qquad g_k^H(p_0,p_1)=0,
\end{equation}
where $g_k^H$ is defined by (\ref{eq:gkH}).
The first equation selects the points from the curve of fixed points, while the second equation guarantees that the solution is also on the bifurcation curve. 
In order to apply the Newton method to the system of equations (\ref{eq:bifPointsScheme}), we need to compute derivatives of $g_k^H$. From (\ref{eq:gkH}) we have
$$
\frac{\partial g_k^H}{\partial p_1}(p_0, p_1) =  \int_0^1 \frac{\partial^2 G_k^H}{\partial p_1^2}\left(p_0,p_1(p_0)+t(p_1-p_1(p_0))\right)tdt. 
$$
If the seed point for the Newton method is quite close to the bifurcation point, we may assume that $p_1(p_0)$  almost constant  and thus
$$
\frac{\partial g_k^H}{\partial p_1}(p_0, p_1) \approx  \frac{\partial^2 G_k^H}{\partial p_1^2}\left(p_0,p_1\right)\int_0^1 tdt = \frac{1}{2}\frac{\partial^2 G_k^H}{\partial p_1^2}\left(p_0,p_1\right). 
$$
The second partial derivative reads
\begin{multline*}
\frac{\partial g_k^H}{\partial p_0}(p_0,p_1) = 
\int_0^1\frac{\partial^2 G_k^H}{\partial p_0\partial p_1}\left(p_0,p_1(p_0)+t(p_1-p_1(p_0))\right)dt +\\ 
\int_0^1\frac{\partial^2 G_k^H}{\partial p_1^2}\left(p_0,p_1(p_0)+t(p_1-p_1(p_0))\right)p_1'(p_0)(1-t)dt.
\end{multline*}
Again, assuming $p_1\approx p_1(p_0)$ we can approximate 
$$
\frac{\partial g_k^H}{\partial p_0}(p_0,p_1) \approx  
\frac{\partial^2 G_k^H}{\partial p_0\partial p_1}\left(p_0,p_1\right)+ \frac{1}{2}\frac{\partial^2 G_k^H}{\partial p_1^2}\left(p_0,p_1\right)p_1'(p_0).
$$
In order to compute $D^2G_k^H$ we need second order derivatives of $\mathcal P^k$ and of the function $p_2^H(p_0,p_1)$ obtained from the Lyapunov-Schmidt reduction --- see (\ref{eq:LSreductionH}). The latest can be computed by differentiation of the identity
$$
\pi_{q_2}\left(\mathcal P^k(p_0,p_1,p_2^H(p_0,p_1),0)\right)\equiv 0.
$$
Summarizing, the Newton-like iteration for equation (\ref{eq:bifPointsScheme}) is given by 
$$
(p_0^{m+1},p_1^{m+1},p_2^{m+1}) = (p_0^{m},p_1^{m},p_2^{m}) -r 
$$
where $r$ is the solution to the linear equation $M\cdot r = b$ with
\begin{eqnarray*}
M&=&\begin{bmatrix} 
  \frac{\partial^2 G_k^H(p_0^m,p_1^m)}{\partial p_0\partial p_1}+ \frac{1}{2}\frac{\partial^2 G_k^H(p_0^m,p_1^m)}{\partial p_1^2}p_1'(p_0^m) 
  & \frac{1}{2}\frac{\partial^2 G_k^H(p_0^m,p_1^m)}{\partial p_1^2} 
  & 0\\ 
  \frac{\partial \widetilde{\mathcal P}(p_0^m,p_1^m,p_2^m)}{\partial p_0}
  & \quad \frac{\partial \widetilde{\mathcal P}(p_0^m,p_1^m,p_2^m)}{\partial p_1} \quad
  & \frac{\partial \widetilde{\mathcal P}(p_0^m,p_1^m,p_2^m)}{\partial p_2}
\end{bmatrix},\\
b &=& \begin{bmatrix} \frac{\partial G_k^H(p_0^m,p_1^m)}{\partial p_1}\\ \widetilde{\mathcal P}(p_0^m,p_1^m,p_2^m) \end{bmatrix}. 
\end{eqnarray*}  
Using the above scheme, finding approximate bifurcation points of halo orbits with accuracy $10^{-60}$ was not a difficult task.
\begin{remark}\label{rem:AppIsoPoint}
 In the computer-assisted proof of Theorem~\ref{thm:main_isochronous} we used similar strategy to localize approximate points of \isoname bifurcations. We solved for zeroes of the function
 \begin{equation*}
  (x,\dot y)\to \left(\pi_{\dot x} \mathcal P(x,0,0,0,\dot y,0),\frac{\partial \pi_{\dot z}\mathcal P}{\partial z}(x,0,0,0,\dot y,0)\right).
 \end{equation*}
\end{remark}

\end{document}